\documentclass[12pt]{article}
\usepackage[english]{babel}
\usepackage{graphicx}
\usepackage{amsmath,amsfonts,amssymb,amsthm,amsopn,amscd}
\usepackage{bbm,wasysym,stmaryrd}
\usepackage{subfigure}
\usepackage{color}
\usepackage{setspace}
\usepackage{authblk}
\usepackage{hyperref}
\usepackage[authoryear]{natbib}
\usepackage{latexsym}

\newcommand{\captionfonts}{\normalsize}

\makeatletter  
\long\def\@makecaption#1#2{%
  \vskip\abovecaptionskip
  \sbox\@tempboxa{{\captionfonts #1: #2}}%
  \ifdim \wd\@tempboxa >\hsize
    {\captionfonts #1: #2\par}
  \else
    \hbox to\hsize{\hfil\box\@tempboxa\hfil}%
  \fi
  \vskip\belowcaptionskip}
\makeatother   

\graphicspath{{Figures}}

\newcommand{\R}{\mathbbm{R}}
\newcommand{\N}{\mathbbm{N}}

\newcommand{\der}[2]{ \frac{\text{d} #1}{\text{d} #2} }  

\newtheorem{Assump}{Assumption}

\theoremstyle{plain}
\newtheorem{theorem}{Theorem}

\newtheorem{lemma}[theorem]{Lemma}

\theoremstyle{definition}
\newtheorem{definition}{Definition}

\begin{document}
	\hspace{13.9cm}1

	\ \vspace{20mm}\\

	\noindent {\flushright {\LARGE On the simulation of nonlinear bidimensional spiking neuron models}}

	\ \\
	{\bf \large Jonathan Touboul$^{\displaystyle 1}$}\\
	{$^{\displaystyle 1}$ NeuroMathComp Laboratory, INRIA/ENS Paris, Paris, France.}\\
%
%
%
%
%

\thispagestyle{empty}
\markboth{J. \textsc{Touboul}}{Simulation of Nonlinear IF Neurons}
\ \vspace{-0mm}\\
%
\begin{center} {\bf Abstract} \end{center}
%
{\it Bidimensional spiking models currently gather a lot of attention for their simplicity and their ability to reproduce various spiking patterns of cortical neurons, and are particularly used for large network simulations. These models describe the dynamics of the membrane potential by a nonlinear differential equation that blows up in finite time, coupled to a second equation for adaptation. Spikes are emitted when the membrane potential blows up or reaches a cutoff $\theta$. The precise simulation of the spike times and of the adaptation variable is critical for it governs the spike pattern produced, and is hard to compute accurately because of the exploding nature of the system at the spike times. We thoroughly study the precision of fixed time-step integration schemes for this type of models and demonstrate that these methods produce systematic errors that are unbounded, as the cutoff value is increased, in the evaluation of the two crucial quantities: the spike time and the value of the adaptation variable at this time. Precise evaluation of these quantities therefore involve very small time steps and long simulation times. In order to achieve a fixed absolute precision in a reasonable computational time, we propose here a new algorithm to simulate these systems based on a variable integration step method that either integrates the original ordinary differential equation or the equation of the orbits in the phase plane, and compare this algorithm with fixed time-step Euler scheme and other more accurate simulation algorithms.   
}

\section{Introduction}
The increasingly finer description of the biophysics of neurons and their ionic channels now makes the biological mechanism of action potential firing understood in great details\footnote{Many questions still remain open, as for instance reviewed in \citep{hille:01}.}. This better understanding allows to propose very precise models of neurons mimicking the dynamics of ionic transfers through the channels. Yet, simple neuron models such as the integrate-and-fire model \citep{lapicque:07,gerstner-kistler:02b} remain very popular in the computational neuroscience community, because they can be simulated very efficiently and, perhaps more importantly, because they are easier to understand and analyze. The drawback is that these simple models cannot account for the variety of electrophysiological behaviors of real neurons
(see e.g. \citep{markram-toledo-rodriguez-etal:04} for interneurons). Recently, several authors introduced two-variable spiking models \citep{izhikevich:04,brette-gerstner:05,touboul:08} which, despite their simplicity, can reproduce a variety of electrophysiological signatures such as bursting or regular spiking. Different sets of parameter values correspond to different electrophysiological classes.

These models seem to provide a compromise between simplicity and versatility. Simplicity, in the sense that it allows analytical studies \citep{touboul:08,touboul-brette:09}, efficient numerical simulations for very large networks simulations \citep{izhikevich-edelman:08}, and a small number of parameters. And versatility, because it is able to reproduce a variety of possible behaviors that neurons present: in \citep{izhikevich:04,touboul:08}, the authors provide simulations corresponding to a wide variety of neuronal behaviors, and simple sets of parameters corresponding to each behaviors can be determined, either analytically \citep{touboul:08, touboul-brette:09} or numerically \citep{naud-macille-etal:08}, and that can be precisely and easily tuned to fit intracellular recordings \citep{clopath-jolivet-etal:07,jolivet-kobayashi-etal:08,jolivet-schurmann-etal:08,rossant-brette:10}. 
These models also aroused the interest of mathematicians, for the nonlinear nature of the dynamics combined with the discrete nature of spikes makes of this class of models an interesting novel mathematical object \citep{touboul:08, touboul-brette:09}. 

These models describe the dynamics of the membrane potential by a nonlinear differential equation, coupled to a second equation for adaptation, and a separate discrete mechanism accounts for the spike emission. In details, these models satisfy reduced equations of the type:
\begin{equation}\label{eq:GeneralModel}
 \begin{cases}
  \der{v}{t} &= F(v)-w+I \\
  \der{w}{t} & = a\,(b\,v -w)
 \end{cases}
\end{equation}
where $a$ and $b$ are non-negative parameters accounting respectively for the time scale of the adaptation variable with respect to the membrane potential's time scale and for the coupling strength between the two variables. The function $F(\cdot)$ is a real function, that satisfies the following assumptions:
\begin{Assump}\label{Assump:Blow}
	$ $
	\begin{enumerate}
		\item $F$ is regular (at least three times continuously differentiable)
		\item $F$ is strictly convex, and $\lim\limits_{x\to -\infty} F'(x) \leq 0$,
		\item There exists $\varepsilon>0$ and $\alpha>0$ for which $F(v)\geq \alpha v^{1+\varepsilon}$ when $v\to \infty$ (we will say that $F$ grows faster than $v^{1+\varepsilon}$ when $v \to \infty$).
	\end{enumerate}
\end{Assump}

Under these assumptions, the membrane potential blows up in finite time for some initial conditions (see \citep{touboul-brette:09} and appendix \ref{sec:theo}). A spike is emitted at the time $t^*$ when the membrane potential $v$ reaches a cutoff value $\theta$ or when it blows up. At this time, the membrane potential is reset to a constant value $c$ and the adaptation variable is updated to $w(t^*)+ d$ where $w(t^*)$ is the value of the adaptation variable at the time of the spike and $d>0$ is the spike-triggered adaptation parameter. Furthermore, if the function $F$ satisfies the assumption:
\begin{Assump}\label{Assump:convergence}
	There exists $\varepsilon>0$ and $\alpha>0$ for which $F(v)\geq \alpha v^{2+\varepsilon}$ when $v\to \infty$ ($F$ grows faster than $v^{2+\varepsilon}$ when $v \to \infty$),
\end{Assump}
\noindent then the adaptation variables converges to a finite value when the membrane potential blows up, which allows replacement of the strict voltage threshold of classical (linear) integrate-and-fire neurons \citep{lapicque:07,stein:67} by a more realistic smooth spike initiation (see e.g. \citep{latham-richmond-etal:00,fourcaud-trocme-hansel-etal:03}).

Among these models, the widely used \emph{quadratic adaptive} model \citep{izhikevich:04} corresponds to the case where $F(v)=v^2$, and has been recently used by Eugene Izhikevich and coworkers \citep{izhikevich-edelman:08} in very large scale simulations of neural networks. This model does not satisfies assumption \ref{Assump:convergence}, and in that case the adaptation variable $w$ blows up at the times of the spikes, which implies that the spike patterns produced by the simulation of this model sensitively depends on the choice of the cutoff, parameter which does not have any biological interpretation (see \citep{touboul:09}). The \emph{adaptive exponential} model \citep{brette-gerstner:05} corresponds to the case where $F(v)=e^v-v$. It has the interest that its parameters can be related to electrophysiological quantities, and has been successfully fit to intracellular recordings of pyramidal cells \citep{clopath-jolivet-etal:07,jolivet-kobayashi-etal:08,badel-lefort-etal:08}. The \emph{quartic} model \citep{touboul:08b} corresponds to the case where $F(v)=v^4+\alpha v$. It has the advantage of being able to reproduce all the behaviors featured by the other two and also self-sustained subthreshold oscillations which are of particular interest to model certain nerve cells.

In these models, the neural code is assumed to be contained in the times of the spikes. Therefore, computing the spike times with accuracy is essential. Moreover, the reset mechanism makes critical the value of the adaptation variable at the time of the spike. Indeed, when a spike is emitted at time $t^*$, the new initial condition of the system \eqref{eq:GeneralModel} is $(c,w(t^*)+d)$. Therefore, this value $w(t^*)$ governs the subsequent evolution of the membrane potential, and hence the spike pattern produced. For instance in \citep{touboul-brette:08,touboul:09,touboul-brette:09}, the authors show that the sequence of reset locations after each spike time shapes the spiking signature of the neuron. They also show that small perturbations can result in dramatic changes, since the spike pattern produced undergoes bifurcations as a function of the parameters. This discontinuity linked with the spike emission might lead to error accumulation on a spike train emission.  

The explosion of the membrane potential variable (and possibly of the adaptation variable) makes the simulation of these models very delicate. Indeed, the explosion implies that the membrane potential $v$, as a function of time, presents singularities at the spike times, where in particular both the membrane potential and the adaptation variables have an infinite derivative (the explosion in finite time always corresponds to an infinite derivative, and equation \eqref{eq:GeneralModel} implies that at these times the derivative of $w$ also tends to infinity\footnote{it can be easily proved that $w(t)$ is always negligible compared to $v(t)$ (i.e. $w(t)=o(v(t))$) at the explosion times}). This property of the model makes the accurate simulation of these models very difficult. For one dimensional models where no adaptation (variable $w$) is taken into account, the spike time can be corrected following \citep{fourcaud-trocme-hansel-etal:03} by computing the explosion time in closed form, which allows a good evaluation of the spike times. Unfortunately, such corrections do not exist for the more biologically realistic bidimensional models with adaptation, and the precise evaluation of spike times and adaptation value at these times necessitate the development a new computational techniques.

In this paper, we start by thoroughly studying the precision of fixed time step methods such as the Euler integration scheme that are the most widely used methods to simulate such equations. We will observe that the error made is unbounded as the threshold $\theta$ is increased, which leads us to propose in section \ref{sec:MyMethod} a new algorithm for accurately and efficiently simulating the solutions of these equations, and in particular two crucial features of the model, namely the spike time and the value of the adaptation variable at these times. The comparison of this new algorithm with some other algorithms is then discussed. All along the paper, we will make use of some theoretical results on the orbits of the dynamical system \eqref{eq:GeneralModel} that are summarized and proved in appendix \ref{sec:theo}. 

\section{Fixed Time Step Simulation Methods}\label{sec:fixedstep}

Most simulation algorithms found in the literature for simulating neuron models of the type of equation \eqref{eq:GeneralModel} involve direct integration of the equations with a fixed time step integration algorithms such as the Euler scheme \citep{izhikevich:03,clopath-ziegler-etal:08,touboul-brette:08}, or more precise Runge-Kutta like schemes. In these simulation algorithms, spikes are emitted when the voltage variable reaches a given threshold $\theta$. The use of such fixed time-step methods for solving this type of blowing-up equations was initially motivated by the observation that the spike time is easily evaluated because of the explosion property of the membrane potential variable $v$, as discussed in \citep{izhikevich:07}, where a typical Euler algorithm is suggested. However, recent research established that in order to accurately simulate these models, one needs to also have an important precision in the evaluation of the adaptation variable at the time of the spike. Indeed, it was shown that this variable governs spiking patterns \citep{touboul:09,touboul-brette:09}, and that small changes in the evaluation of this variable can quantitatively and qualitatively modify them.

The crucial question that this observation raises is that beyond the necessary precision in the spike time, the model requires to accurately evaluate the value of the simulated adaptation variable at the time of the spike. In this view, though usual fixed time-step methods might provide a fair precision on the evaluation of the spike time (we will show in the sequel that even this evaluation is very imprecise with such fixed time-step methods, involving unbounded errors as the spiking cutoff is increased), they are very imprecise for calculating the value of the adaptation variable at the time of the spike as we rigorously show in the present section. This imprecision will substantially modify the value of the time of the next spikes, and this effect will be increasingly important as the number of spikes emitted increases. Indeed, because of the discontinuity induced by the spike emission and the reset, the numerical errors will accumulate. We first analyse the precision for the computation of the first spike, before addressing the question of how do these errors accumulate.

\subsection{Precision of the computation of the first spike time and adaptation variable at this time}
We now go into the mathematical details of the precision analysis of such fixed time step methods for simulating the solutions of equation \eqref{eq:GeneralModel}. We concentrate on the simpler and most widely used fixed-time step Euler scheme, with threshold $\theta$ and time step $\tau$. The numerical solutions computed through such numerical procedure are solutions of the recursion:
\begin{equation}\label{eq:Euler}
	\begin{cases}
		v_{n+1} &= v_n + \tau\,(F(v_n)-w_n + I(n\tau))\\
		w_{n+1} &= w_n + \tau\,a\,(b\,v_n-w_n)
	\end{cases}
\end{equation}
starting form a given initial condition $(v_0,w_0)$. We denote $X_n:=(v_n,w_n)$ and $\Phi$ the map such that $X_{n+1}=\Phi(X_n)$ given by the recursion \eqref{eq:Euler}. We define a zone of the phase plane $Z^*$ as a subset of the spiking zone defined in \citep{touboul-brette:09} which will be of particular interest in the sequel. 

\begin{definition}\label{def:Z}
	Let $Z(V)$ be the region of the phase plane $(v,w)$ defined as:
		\[Z(V)=\{(v,w) \in \R^2\;;\;\; v \geq V \textrm{ and } w \leq  b\,v\}.\]
	Note that $Z(V)$ constitute a decreasing family of nested sets. 
	
	We further define, when the input current $I(t)\geq I^*$ is lowerbounded,
	\[Z^*:=Z(I^*,b)=\cup_{V \geq v_+(I^*,b)} Z(V),\]
	where $v_+(I^*,b)$ is the largest fixed point of the system \eqref{eq:GeneralModel} when it exists, and $-\infty$ if it does not (see appendix \ref{sec:theo}). We have if $m(b)$ denotes the minimal value of $F(v)-b\,v$:
	\begin{itemize}
	 		\item If $I^*>-m(b)$, then $Z(I^*,b)=\{(v,w) \in \R^2\;;\;\; w \leq b\,v\}$.
	 		\item If $I^*<-m(b)$, then $Z(I^*,b)=\{(v,w) \in \R^2\;;\;\; v \geq v_+(I^*,b) \textrm{ and } w \leq b\,v\}$
 		\end{itemize}
\end{definition}

When one does not consider any threshold on the dynamics of the membrane potential variable $v$, any spiking trajectory will eventually belong to $Z^*$ after a transient phase. 

\begin{lemma}\label{lemma:EulerMonot}
	Let us assume that $(v_0,w_0)\in Z^*$ and $\tau \, a<1$. Then for all $n\geq 0$, we have $X_n\in Z^*$ and both sequences $(v_n)$ and $(w_n)$ are strictly increasing and finite. Moreover, $\lim_{n\to\infty} v_n = \infty$. 
\end{lemma}

\begin{proof}
	We reason by induction on $n$. We know that $X_0=(v_0,w_0)\in Z^*$. Let us assume that for some $n\geq 0$, we have $X_n=(v_n,w_n) \in Z^*$. The finiteness of $X_{n+1}=\Phi(X_n)$ stems straightforwardly from the finiteness of $X_n$. 
	
	Moreover, since $X_n\in Z^*$ implies that $w_n \leq b\,v_n$ and $v_n \geq v_+(I^*,b)$ when it exists, we have:
	\[F(v_n)-w_n+I(n\tau) \geq F(v_n) - b\,v_n + I^* >0\]
	 and because of the convexity assumption made of $F$ we have $F(v)-b\,v+I^*>0$ for all $v> v_+(I^*,b)$ when it exists, or for all $v$ when no fixed point exists. This directly implies that $v_{n+1}>v_n$. Moreover, since in $Z^*$, we have $w<b\,v$, it is immediate to show that $w_{n+1}>w_n$, implying that both sequences are increasing.
	
	We therefore have $v_{n+1}>v_n \geq v_0 > v_+(I^*, b)$. Therefore, showing that $X_{n+1}=\Phi(X_n)\in Z^*$ only amounts proving that $w_{n+1}<b\,v_{n+1}$. We have:
	\begin{align*}                                                                                                                                                                                                                                                                                                                                                                                                                                                                                                                                                                                                                                                                                                                   
		w_{n+1}-b\,v_{n+1} &= (w_n - b v_n) + \tau [a\,(b\,v_n-w_n) -b\,(F(v_n)-w_n+I(n\,\tau))]\\
		&= \underbrace{(1-\tau\,a)}_{>0} \underbrace{(w_n-b\,v_n)}_{<0} - \underbrace{b\,\tau\,(F(v_n)-w_n +I(n\tau))}_{>0}\\
		&<0.
	\end{align*}
	Therefore, for all $n\geq 0$ we have $X_n \in Z^*$.
	
	To end the proof of the lemma, we therefore need to prove the unboundedness property of $v_n$. Since we have $v_{n+1}-v_n \geq \tau(F(v_n)-b\,v_n + I^*) > \tau (F(v_0)-bv_0 + I^*)=:K_0$ wich is strictly positive constant, we have for all $p\in \N$ the inequality $v_p \geq v_0 + p\, K_0$, and therefore for all $M>0$ there exists $n\in \N$ such that $v_n>M$ which ends the proof of the lemma.
\end{proof}

We therefore observe that for any initial condition in the spiking zone, the adaptation variable of the solution of the Euler simulation scheme diverges (i.e. $v_n \to \infty$ when $n\to \infty$). However, for any $n\in \N$, we have shown that $v_n$ is finite, hence the Euler scheme does not blow up in finite time. This fact implies in particular that the error made on the evaluation of $v$ is {\it a priori}  unbounded if there is no cutoff in the simulation: solutions of \eqref{eq:GeneralModel} starting from some initial condition $(v_0,w_0)\in Z^*$ blow up in finite time (say $t^*$) and the approximated value of $v$ at time $t^*$ computed from the Euler scheme is finite, producing an infinite error on the evaluation of $v$. 

However, as we already mentioned, numerical methods truncate the solution as soon at the value of the adaptation variable reaches a given threshold (or cutoff) $\theta$. The spike time corresponds to the first time the simulated trajectory exceeds $\theta$, and is then reseted. For $v\leq \theta$, the vector field Lipschitz-continuous, and it is known from standard numerical analysis theory that the thresholded discretization scheme is of order one in $\tau$, meaning that the error is bounded by a constant depending on the parameters of the model multiplied by $\tau$. Ensuring an absolute given precision on the spike time and the adaptation variable at this time therefore amounts controlling the value of the constant involved in the error, which is what we now do.

Let us now have a closer look at the absolute error made in the thresholded discretized solution. We define the functions $\alpha(\cdot)$ and $\beta(\cdot)$ corresponding to the first order error in $\tau$ of the discretization of $v$ and $w$ as:
\[
\begin{cases}
	v_n=v(n\tau) + \tau\, \alpha(n\,\tau) + O(\tau^2)\\
	w_n=w(n\tau) + \tau\, \beta(n\,\tau)  + O(\tau^2)
\end{cases}
	\]
We have: \[\alpha(0)=\beta(0)=0,\]
and moreover (see e.g. \citep{hairer-lubich:83,isaacson-keller:94,sanz-serna:86}):
\begin{align*}
	v_{n+1} - v_n & = v(n\tau+\tau) -v(n\tau) +\tau (\alpha (n\tau +\tau)-\alpha(n\tau))+O(\tau^2) \\
	& = \tau v'(n\tau) +\frac{\tau^2}{2} v''(n\tau) + \tau^2 \alpha'(n\tau) + O(\tau^3)\\
	& = \tau\,(F(v(n\tau) )-w(n\tau) + I(n\tau)) +\frac{\tau^2}{2} v''(n\tau)+ \tau^2 \alpha'(n\tau) + O(\tau^3).
\end{align*}
By the recursion relationship between $v_n$ and $v_{n+1}$ we have:
\begin{align*}
	v_{n+1} - v_n & = \tau \left (F(v_n)-w_n+I(n\tau)\right)\\
	&=\tau \left(F(v(n\tau) +\tau \alpha(n\tau)+O(\tau^2))-w(n\tau)-\tau \beta(n\tau) + I(n\tau) + O(\tau^2)\right)\\
	&=\tau \left(F(v(n\tau)) + F'(v(n\tau))\tau \alpha(n\tau)-w(n\tau)-\tau \beta(n\tau) + I(n\tau)\right) + O(\tau^3)
\end{align*}
Equalizing the two expressions, we get:
\[\alpha'(n\tau) = F'(v(n\tau))\,\alpha(n\tau) - \beta(n\tau)-\frac 1 2 v''(n\tau)\]
proceeding in the exact same fashion for the difference $w_{n+1}-w_n$, we are led to the following system of ordinary differential equations on the first order error of the Euler scheme:
\begin{equation}\label{eq:errorEulerSimple}
	\begin{cases}
	\alpha'(t) &= F'(v(t))\,\alpha(t) -\beta(t) -\frac 1 2 v''(t)\\
	\beta'(t)  &= a\,(b\,\alpha(t)-\beta(t)) - \frac 1 2 w''(t)
\end{cases}
\end{equation}

Using the equations \eqref{eq:GeneralModel} governing $v$ and $w$, we obtain the equations:
\begin{equation}\label{eq:errorEulerComplex}
	\begin{cases}
	\alpha' &= F'(v)\,\alpha -\beta -\frac 1 2 \Big\{F'(v)\left(F(v) -w+I\right)-a\,(b\,v-w) + I'\Big\}\\
	\beta'  &= a\,(b\,\alpha-\beta) - \frac a 2 \Big\{b\,(F(v)-w+I)-a\,(b\,v-w)\Big\}
\end{cases}
\end{equation}
It is easy to prove using Gronwall's theorem along the same lines as in \citep{touboul:09} that both errors tend to infinity when no threshold is considered, as we also noted from the fact that the continuous solution $v(t)$ blows up in finite time and the discretized version $v_n$ remains finite for all $n\in \N$. 

Let us consider a spiking trajectory. From theorem \ref{theo:blowup} of appendix \ref{sec:theo} and in lemma \ref{lemma:EulerMonot}, we know that as soon as the trajectory enters the spiking zone $Z^*$, the error can be parametrized as a function of the membrane potential $v(t)$. In the sequel, instead of studying the error functions $\alpha(t)$ and $\beta(t)$, it will appear particularly convenient to study the composed applications $A(v)=(\alpha \circ T)(v)=\alpha(T(v))$ and $B(v)=(\beta \circ T)(v)=\beta(T(v))$ where $T(v)$ is defined in theorem \ref{theo:blowup} of appendix \ref{sec:theo} as the inverse of the function $t\mapsto v(t)$ and $\circ$ denotes the composition of applications. These two functions satisfy the equations:
\begin{equation*}
		\begin{cases}
		\der{A}{v} &= \frac{F'(v)}{F(v)}\,A - \frac{1}{F(v)} \, B -\frac 1 {2\,F(v)} \Big\{F'(v)\left(F(v) - W+I\circ T\right)-a\,(b\,v-W) + I'\circ T\Big\}\\
		\der{B}{v}  &= \frac{a}{F(v)}\,(b\,A-B) - \frac a {2\,F(v)} \Big\{b\,(F(v)-W+I\circ T)-a\,(b\,v-W)\Big\}
	\end{cases}
\end{equation*}
These equations, similarly to initial model's equations, are quite hard to solve in their general setting. However, using the fact that $W<b\,v$ and that on a spiking trajectory $W=o(v)$ and similarly $\beta=o(v)$, the behavior of the error close to a relatively high threshold can be approximated by the solution of the equations:
\begin{equation}\label{eq:errorsimplified}
\begin{cases}
	\der{A}{v} &= \frac{F'(v)}{F(v)}\,A - \frac 1 2 F'(v)\\
	\der{B}{v} &= \frac{a}{F(v)}\,(b\,A - B) - \frac 1 2 a\,b\\
\end{cases}
\end{equation}
These equations can be easily understood heuristically. Indeed, since we known that at the times of the spikes, the membrane potential blows up and the adaptation variable is dominated by the value of the membrane potential variable, the error made on the membrane potential evaluation essentially stems from the divergence membrane potential variable, that asymptotically satisfies an equation of type $\dot{y}=F(y)+I$. This is why the first equation is exactly the same as the one for the error of fixed-time step methods for this one-dimensional equation. The error $A$ computed is the term that mostly disturbs the evaluation of the adaptation variable, and acts on the second equation of \eqref{eq:errorsimplified} as an independent input function. This is coherent with the heuristic argument that the main part of the errors made on the estimation of the adaptation variable at the time of the spike is therefore related to the imprecision in the evaluation of the membrane potential variable close to the explosion. 

The study of the first equation of \eqref{eq:errorsimplified} and of the threshold crossing time is quite straightforward at this point. The first equation of \eqref{eq:errorsimplified} is integrated as:
\[A(v) = \frac{1}{2}\Big(\log\big(F(v_0))-\log(F(v)\big)\Big)\,F(v)\]
which yields for the error on the adaptation variable:
\[B ( v ) = -\frac{a\,b}{2}  {\rm e}^{-\int_{v_0}^{v}{\frac {a}{F( u ) }}du} \int_{v_0}^v {\rm e}^{\int_{v_0}^u \frac{a}{F(u_1)} du_1} ( \log( F( u)) -\log (F(v_0))) {du}\]
and the error made on the evaluation of the adaptation variable because of the discretization at the threshold crossing time is therefore equal to $\vert B(\theta) \vert$.
It is now easy to instantiate the models and find an approximation of the error:
\renewcommand{\theenumi}{(\roman{enumi})}
\begin{enumerate}
	\item For $F(v)$ having a polynomial dominant term $v^m$ (which is for instance the case of Izhikevich quadratic model and the quartic model), we have:
	\[
	\begin{cases}
		A(v) &= \displaystyle{-\frac m 2 \;v^m \log\left(\frac{v}{v_0}\right)}\\
		B(v) &= \displaystyle{-\frac {a\,b\,m} 2 \;e^{\frac{v^{-m+1}a}{m-1}}\int_{v_0}^v e^{-\frac{u^{-m+1}a}{m-1}}\log\left(\frac{u}{v_0}\right)\,du}
	\end{cases}
	\]
	\item For $F(v)$ equivalent to an exponential function (as it is the case in the adaptive exponential model), we have the error estimate:
	\[
	\begin{cases}
		A(v) &= \displaystyle{-\frac 1 2 (v-v_0) \, e^{v}}\\
		B(v) &= \displaystyle{-\frac {ab} 2 \; e^{a\,e^{-v}} \int_{v_0}^v e^{-a\,e^{-u}}\left(u-v_0\right)\,du}
	\end{cases}		
	\]
\end{enumerate}
These functions are plotted as a function of $v$ in Figure \ref{fig:Errors}. They are both unbounded and diverging as $\theta$ increases. 
\begin{figure}
	\centering
	\includegraphics[width=.4\textwidth]{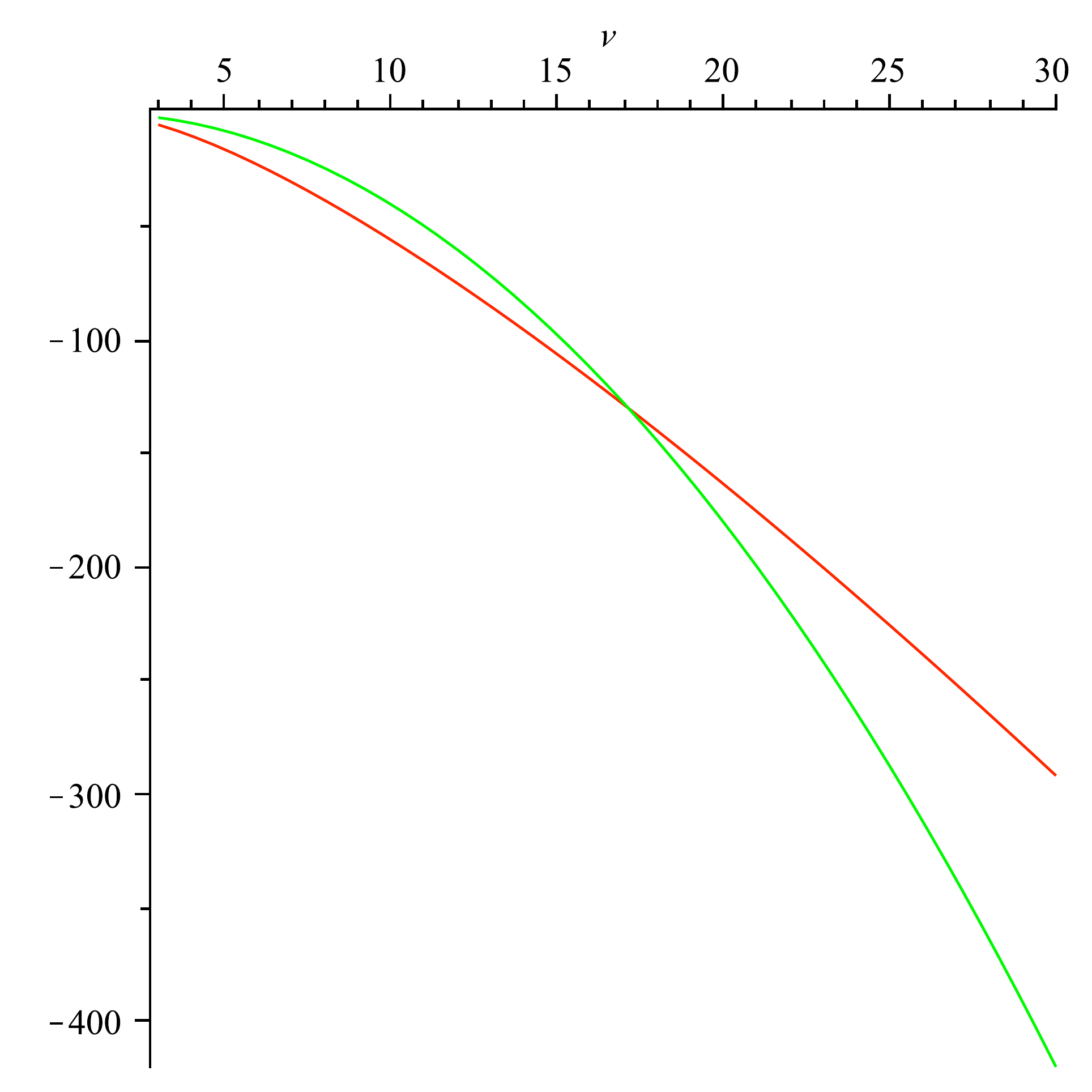}
	\caption{Error made on the estimation of the adaptation variable as a function of the threshold used. We see a fast divergence of the error. Red solid curve: quartic case, Green dotted curve: exponential case}
	\label{fig:Errors}
\end{figure}
This result implies that the greater the cutoff is the more imprecise the evaluation of the adaptation variable at the cutoff time is, and that the error diverges as the cutoff value is increased. Hence the use of fixed time step methods to approximate such nonlinear integrate-and-fire models leads to substantial quantitative errors on the evaluation of $v$ and $w$ that might change the qualitative behaviors, as we will further address in section \ref{sec:Results}. We also note that the error made is always negative: the evaluated $w$ at the time of the spike will be systematically smaller than the actual value of the adaptation variable. 

\medskip

We now adress the following question: is the spike time accurately evaluated by this procedure? The answer to this question was understood to be positive, and the argument given was the divergence of the membrane potential at the time of the spike. However, as we rigorously prove in the sequel, this is not exact. Indeed, the explosion only concerns the differential equation, whereas the numerical procedure does not blows up in finite time, but diverges when time tends to infinity. This induces increasingly large (and even infinite) errors on the evaluated membrane potential $v$ that produces systematic errors at the threshold crossing time, and these errors are unbounded as the cutoff is increased. The proof of this fact is quite difficult to produce in the full system. However, as already discussed, because of the fact that during the explosion we have $w=o(v)$, we will restrict the study to the comparison of the recursion $y_{n+1}=y_n + F(y_n)$ and the blowing up solution of $y'=F(y)$, and only for power and exponential $F$ functions, which cover most practical cases.

Let us treat in parallel the power and the exponential cases. Let $m>1$ be a real number, and consider respectively the ordinary differential equations:
\begin{minipage}{0.4\textwidth}
	\[\begin{cases}
	y_p'&=y_p^m\\
	y_p(0)&=y^{p}_0
\end{cases}\]
\end{minipage}
\qquad
\begin{minipage}{0.4\textwidth}
	\[\begin{cases}
	y_e'&=e^{y_e}\\
	y_e(0)&=y^{e}_0
\end{cases}\]
\end{minipage}\\
The solution of these ordinary differential equation reads (respectively):
\begin{align*}
	y_p(t) &= \left[(y^p_0)^{1-m}-(m-1)\,t\right]^{1/(1-m)}\\
	y_e(t) &= y^e_0-\log(1-t\,e^{y^e_0})
\end{align*}
and blow up at time $t^*=(y^p_0)^{1-m}/(m-1)>0$ (resp $t^*=e^{-y^e_0}$). 
The numerical solution from Euler discretization satisfy the recurrence equation: 
\[
\begin{cases}
	y^p_{n+1} &= y^p_n + \tau (y^p_{n})^{m}\\
	y^e_{n+1} &= y^e_n + \tau e^{y^e_{n}}
	\end{cases}
\]
with initial condition $y^p_0$ and $y_0^e$ respectively. The approximation error produced by the discretization algorithm was previously computed, and are equal to:
\[\begin{cases}
A^p(t)&=-\frac m 2 \;y^p(t)^m \log\left(\frac{y^p(t)}{y^p_0}\right)\\
A^e(t)&=-\frac 1 2 \;(y^e(t) - y_0^e)\;e^{y^e(t)}
\end{cases}\]
for $0\leq t < t^*$, and therefore we have:
\[\begin{cases}
y^p_n&=y^p(t)-\tau \frac m 2 \;y^p(t)^m \log\left(\frac{y^p(t)}{y^p_0}\right)+O(\tau^2)\\
y^e_n&=y^e(t)-\tau \frac 1 2 \;(y^e(t) - y_0^e)\;e^{y^e(t)} + O(\tau^2)
\end{cases}\]
when $\tau \to 0$ and for $0\leq t < t^*$. Let us now define $z^{p}(k)$ (resp. $z^{e}(k)$) the continuous variable linear interpolation of $y_n^p$ (resp $y_n^e$) between two integers $n$ (we have $y^{e}(n)=z^{e}(n)$ for all $n\in \N$ and between two consecutive integers the function $k\in [n,n+1]\mapsto z^{e}(k)$ is linear, and similarly for the polynomial case with indexes $p$). The relation between $z^{p}(k\tau)$ and $y^{p}(t)$ (resp.  $z^{e}(k\tau)$ and $y^{e}(t)$) is the same as the relation between $y^{p}(t)$ and $y^{p}_n$ (resp. $y^{e}(t)$ and $y^{e}_n$) since the errors made by a linear interpolation in an interval of length $\tau$ are of order $O(\tau^2)$. Let us now define $n_{e}^*\in\R $ and $n_{p}^*\in\R $ such that $z_e(n_{{e}}^*)=z_p(n_{{p}}^*)=\theta$. We aim at comparing $n_{e}^*$ and $n_{e}^{**}$ defined by $y^{e}(n_{e}^{**}\tau)=\theta$ (and similarly in the polynomial case). We have:
\[\begin{cases}
n_p^{**} &= \frac{(y_0^p)^{1-m}}{(m-1)\tau} - \frac{1}{(m-1)\theta^{m-1}\tau}\\
n_e^{**} &=\frac 1 \tau \left(e^{-y_0^e}- e^{ - \theta}\right) 
\end{cases}
\]
Let now $Z_{u}=y_u(n^{*}\tau)$ ($u=e,p$), we have:
\[
\begin{cases}
	\theta &=Z_p  -\frac 1 2 \tau\, m \,Z_p^m \log(\frac{Z_p}{y_0^p})\\
	\theta &=Z_e  -\frac \tau 2 (\,Z_e - y_0^e)e^{Z_e}
\end{cases}
\]
which can be solved up to the second order (by identifying the constant and the linear in $\tau$ coefficients of the solution):
\[
\begin{cases}
	Z_p &= \theta+\frac {m\tau} 2 \, \theta^m \log (\frac{\theta}{y_0^p})+O(\tau^2)\\
	Z_e &= \theta+\frac \tau 2 (\theta-y_0^e)e^{\theta} +O(\tau^2)
\end{cases}
	\]
and eventually inverting the expression of $y(t)$ one obtains:
\[
\begin{cases}
	n_p^* &= n_p^{**}+ \frac m 2 \theta^m\log(\frac{\theta}{y_0})\\
	n_e^* &= n_e^{**} + \frac 1 2 (\theta - y_0^e)
\end{cases}
\]
Therefore, there is a systematic delay of the estimated spike time produced compared to the actual spike time, that increases as the cutoff is increased. In order to achieve precision on the evaluation of this variable also, one therefore needs to use small time steps. 

\medskip

Therefore, we showed, by analogy to one dimensional ordinary differential equation that capture the explosion of the membrane potential at the time of the spike, that a systematic error is made on the evaluation of the spike time, which is of first order in $\tau$ with a bounded coefficient that diverges when the cutoff is increased, and most important that the errors made in the estimation of the adaptation variable at the time of the spike, of order one in $\tau$ also, have a coefficient that is of large amplitude and that is fast diverging when the cutoff is increased. This implies that very small time steps are necessary to achieve a given precision of the simulation, which leads to excessive precision in region of parameters where the vector field is smooth and that questions the efficiency of the obtained simulation algorithm. 

We observe that the faster the divergence of the membrane potential, linked with the growth of the function $F(v)$, the larger the error made on the adaptation variable at the cutoff time, and the smaller the error made on the spike time are.

\subsection{Error Accumulation during a spike train}
We have seen that starting from the same initial condition, a fixed time step Euler scheme produces errors on the evaluation of the spike time and on the value of the adaptation variable at this time. Therefore, when simulating the emission of a spike train, an additional deviation between the numerical spike train computed and the actual spike train appears linked with a shift in the spike time and on the value of the adaptation variable at these times, and as we show here, since the map is expanding, the distance between two trajectories increases as time goes by, which increases numerical errors produced. Indeed, assume that the actual system spikes at time $t^*$ and the simulated system at time $t^{**}$, and that the actual value of the adaptation at this time $w^*$ is approximated by $w^{**}$. We therefore need to compare the solutions of equation \eqref{eq:GeneralModel} with different initial conditions $(t^*,c,w^*)$, solution we denote $(v_1,w_1)$, and $(t^{**},c,w^{**})$, denoted $(v_2,w_2)$. We have for $t>\max(t^{*},t^{**})$:
\[
\begin{cases}
	\displaystyle{v_1(t)-v_2}(t)& =\displaystyle{\int_{t^*}^{t^{**}} F(v_1(s))-w_1(s)+I(s)\,ds} \\
	& \qquad + \displaystyle{\int_{t^{**}}^t F(v_1(s))-F(v_2(s))-(w_1(s)-w_2(s)) \, ds}\\
	\displaystyle{w_1(t)-w_2(t)} &= \displaystyle{(w^*e^{-a\,(t^{**}-t^{*})}-w^{**})e^{-a\,(t-t^{**})} + \int_{t^*}^{t^{**}} a\,b\,v_1(s)e^{-a\,(t-s)}\,ds} \\
	&\qquad + \displaystyle{\int_{t^{**}}^t a\,b\,(v_1(s)-v_2(s))e^{-a\,(t-s)}\,ds}
\end{cases}
\]
The error on the value of the adaptation variable $w$ is therefore governed by the error made on the membrane potential variable, and this difference is unboundedly large close to the spike emission time where the flow is very expanding because of assumption~\ref{Assump:Blow}(iii). Indeed, for all $K>0$ there exists $B>0$ such that for all $x,y$ greater than $B$, $\vert f(x)-f(y) \vert \geq K \vert x-y\vert$). Therefore the errors on the computation of the time and adaptation variable are expanded if the threshold is large enough, and there is no way to control it since this divergence is linked with the divergence of the solutions of equation \eqref{eq:GeneralModel} with different initial conditions\footnote{Note also that this effect is increasingly important as the cutoff is increased. The complete proof of this result involves Gronwall theorem, and can be related to the contraction theory (see e.g.\cite{slotine-li:91,lohmiller-slotine:98})}

\medskip

We therefore showed in this section that the Euler fixed time step methods yielded important errors on the computation of the spike times and on the value of the adaptation variable at these times. The results still hold for any fixed time step algorithm since these are linked with the very nature of the problem, namely the explosion at the time of the spike and the reset discontinuity, in addition to the infinite slope of the adaptation variable at the spike time. These issues motivate the introduction of an integration scheme adapted to the nature of the equations we simulate.

\section{Fixed Absolute Precision Simulation Algorithm}\label{sec:MyMethod}
The explosion, and the infinite slope of the adaptation variable at the time of the spike make the accurate simulation of the spike very delicate, and usual fixed time-step methods necessitate small time steps in order to achieve a precise evaluation of the spike time and of the adaptation variable at these times, which results in excessive precision in regions where the vector field varies slowly. 

In this section we propose a novel algorithm in order to overcome these difficulties. The heuristic idea behind this approach is the following: in order to overcome the difficulty of computing close to the explosion the diverging membrane potential variable as a function of time, we propose to turn the picture around and to simulate time as a function of the membrane potential. The simulation algorithm we propose is based on a simulation of the trajectories in the phase plane as studied in appendix \ref{sec:theo}, in particular in theorem \ref{theo:blowup}. More specifically, the algorithm is based on the local inversion theorem. In details, let $M>0$ be a fixed constant, $t_0\in \R$ a given time and assume that $(v(t),w(t))$ is an orbit of \eqref{eq:GeneralModel} containing the point $(t_0, v_0,w_0)$. Provided that $v'(t_0)=F(v_0)-w_0+I(t_0)\geq M$, there exists $\delta > 0$ such that $t\mapsto v(t)$ is invertible on $[t_0-\delta, t_0+\delta]$. Let us denote by $v\mapsto T(v)$ the inverse of the function $t\mapsto v(t)$ and by $W(v)$ the composed application $w\circ T(v)=w(T(v))$. These functions are differentiable in a neighborhood of $(v_0,w_0)$ and the differential reads:
\begin{equation}\label{eq:WT}
	\begin{cases}
		\displaystyle{\der{T}{v}} &= \displaystyle{\frac{1}{F(v)-W+I(T)}}\\
		\\
		\displaystyle{\der{W}{v}} &= \displaystyle{\frac{a\,(b\,v-W)}{F(v)-W+I(T)}}
	\end{cases}
\end{equation}
This inversion is valid as long as $v'(t)$ is invertible, i.e. as long as $F(v(t))-w(t)+I(t)>0$. We have 
\[v''(t)=F'(v(t))(F(v(t))-w(t)+I(t))-a\,(b\,v(t)-w(t))+I'(t)\]
and therefore
\[\vert v''(t) \vert \leq \vert F'(v)(F(v)-W(v)+I(T(v)))- a\,(b\,v-W(v))+I'(T(v))\] 
with $F(v_0)-w_0+I(t_0)\geq M$. We can therefore define $\delta'>0$ as a function of this second derivative, and such that the description \eqref{eq:WT} is valid on $[v_0-\delta', v_0 + \delta']$. This choice of $\delta'$ will be further discussed in the description of the algorithm. 

Moreover, it is important to note that as soon as the trajectory enters the spiking zone, in particular when it enters the zone $Z^*$, this description of the trajectories in the phase plane is valid until the spike emission (explosion of the membrane potential). More precisely, we demonstrate in appendix \ref{sec:theo} the following:
\begin{theorem}\label{theo:blowup}
	Assume that the input current $I(t)$ depends on time, and moreover that $t\mapsto I(t)$ is lower bounded (at least on the time interval considered $(\tau_0,\tau_1)$), i.e. $I(t) \geq I^*$ for any $t\in (\tau_0,\tau_1)$. Then we have:
	\renewcommand{\theenumi}{(\roman{enumi})}
	\begin{enumerate}
		\item $Z(I^*,b)=:Z^*$ is stable under the flow of the equation
		\item Let $(v_0,w_0) \in Z^*$, and denote $(v(t),w(t))$ the solution of equations \eqref{eq:GeneralModel} having initial condition $(v_0,w_0)$ at time $t_0 \in (\tau_0,\tau_1)$. Then $t \mapsto v(t)$ is strictly increasing and has a strictly positive derivative, and therefore is smoothly invertible. Its inverse function $T(v)$ is differentiable, hence so is $W(v)=w(T(v))$, and these variable satisfy the nonlinear differential equation \eqref{eq:WT}. The explosion time (resp. the value of the adaptation variable at this time) is the limit of $T(v)$ (resp. $W(v)$) when $v\to \infty$.
	\end{enumerate}
\end{theorem}

Therefore the description given by equations \eqref{eq:WT} is valid in particular during the explosion of the membrane potential variable. Both the time of the spike and the value of the adaptation variable at this time can be accurately and easily simulated by solving the smooth ordinary differential equations \eqref{eq:WT}, that governs the trajectory $W(v)$ of the adaptation variable in the phase plane and $T(v)$ the inverse function of $t\mapsto v(t)$.  The description of the trajectories in the phase plane allows simulating well-behaved equations during this phase where the Euler scheme was particularly sensitive. This description is the cornerstone of the algorithm we develop.

\medskip

The discrete integration scheme makes use of this convenient description of the trajectory in the phase plane in regions of the phase plane where the function $t\mapsto v(t)$ is smoothly invertible, and in the rest of the phase plane close to the $v$-nullclines, simulates the solutions of the equations \eqref{eq:GeneralModel} with a fixed time step Euler scheme.

We propose two algorithms based on this idea. The first one is a fixed integration step method and the second one a fixed precision method with adaptive integration steps.

\subsection{Fixed integration-step method}
Let us assume that we have defined $M >0 $ a real parameter, and let $dt>0$ and $dv>0$ be respectively a time step and a space step that will be used to numerically compute the solution of equation \eqref{eq:GeneralModel} and \eqref{eq:WT}.

Starting from an initial condition $(v_0,w_0)$ at time $t_0$, we build iteratively a sequence $(t_n,v_n,w_n)$ approximating the orbit $(t,v(t),w(t))$ solution of the continuous ordinary differential equations \eqref{eq:GeneralModel}. Given $(t_n,v_n,w_n)$ the numerical solution computed at iteration $n$, we compute  $t_{n+1}$, $v_{n+1}$ and $w_{n+1}$ as follows:

\begin{enumerate}
	\item If $ \vert F(v_n)-w_n+I(t_n)\vert < M$, then we update the values of the triplet as:
	\[ \begin{cases}
				v_{n+1} &= v_{n} + dt \, (F(v_n)-w_n+I(t_n)) \\
				w_{n+1} &= w_{n} + dt \, a \, (b\,v_n-w_n) \\
				t_{n+1} &= t_{n} + dt
			\end{cases}
	\]
corresponding to the integration of equation \eqref{eq:GeneralModel} with a constant time step $dt$. This phase of the integration algorithm is referred as the time-integration phase.
	\item if $\vert F(v_n)-w_n+I(t_n)\vert > M $, then we update the triplet as:
	\[ \begin{cases}
				v_{n+1} &= v_n + dv\\
				w_{n+1} &= w_n + dv\, \frac{a\,(b\,v_n-w_n)}{F(v_n)-w_n+I(t_n)}\\
				t_{n+1} &= t_n + dv\, \frac{1}{F(v_n)-w_n+I(t_n)}
			\end{cases}
	\]
which corresponds to the integration of \eqref{eq:WT} with constant $v$-step $dv$. This phase of the integration algorithm is referred as the phase-plane integration phase. 
\end{enumerate}

Intuitively, this method consists exactly as announced in solving \eqref{eq:GeneralModel} in the regions of the phase plane where the vector field is of small amplitude and solving  \eqref{eq:WT} in the regions of phase plane where the vector field of the dynamical system given by \eqref{eq:GeneralModel} is of large amplitude and yield divergence in finite time of the solutions. 

Let us now analyse this algorithm and control the precision of the solution computed. First of all, we analyse the solutions of the algorithm in spiking regions when no threshold is taken into account.

\begin{lemma}\label{lem:DiscreteBlowUp}
	For any initial condition $(v_0,w_0)$ in the spiking zone, the variable $v_n$ blows up in finite time. 
\end{lemma}
\begin{proof}
	 Let $(v_0,w_0) \in Z^*$, we have for any $n$, $(v_n,w_n) \in Z^*$. After a finite number of iterations $K$, we will have for all $n\geq K$ $F(v_n)-w_n+I(t_n) > M$. Indeed, while $(v_n,w_n)\in Z^*$, lemma \ref{lemma:EulerMonot} implies that $v_n$ is unboundedly increasing and moreover in that zone we have 
	\[F(v_n)-w_n+I(n\tau) \geq F(v_n)-b\,v_n + I^*.\]
	Since $v_n$ is unboundedly increasing when simulated with Euler scheme, so is $F(v_n)-b\,v_{n}+I(n\tau)$ and therefore there exists $K\in \N$ such that $M<F(v_K)-b\,v_K+I^*\leq F(v_K)-w_K+I(K\tau)$. 
	
	The next step $v_{K+1},w_{K+1}$ is hence computed using the simulation in the phase plane algorithm. Using the same argument as in the proof of lemma \ref{lemma:EulerMonot}, we obtain that the sequence $(v_n,w_n)$ stays in a zone where $F(v)-w+I>M$ and $w<b\,v$, and therefore for all $n\geq K$, the simulation algorithm used is the phase plane algorithm. 

	We therefore obtain that $v_n$ tends to infinity since $v_{n+1}=v_{n}+dv$. We now show that the sequence $t_n$ converges to a finite limit. Indeed, we have for all $n\geq K$:
	\begin{align*}
		t_n &=t_K + dv \sum_{k=K+1}^n \frac 1 {F(v_n)-w_n+I(t_n)} \\
				&\leq t_K+ dv \sum_{k=K+1}^n \frac 1 {F(v_K+(n-K)\,dv)-b\,(v_K+(n-K)\,dv)+I^*}.
	\end{align*}
 The series converges when $n\to \infty$ because of the hypothesis that $F$ grows faster than $v^{1+\delta}$ for some $\delta>0$, by a simple sum/integral comparison. 
\end{proof}

In order to compute the precision of the algorithm in the evaluation of the spike time and on the value of the adaptation variable at the time of the spike, we now consider a thresholded version of the simulation algorithm consisting as usual in introducing a cutoff $\theta$. The simulation algorithm is run as long as $v_n<\theta$. When $v_n$ exceeds $\theta$, we instantaneously reset the membrane potential and update the adaptation variable by defining:
\[\begin{cases}
	v_{n+1}&= c\\
	w_{n+1}&= w_n+d\\
	t_{n+1}&= t_n
\end{cases}\]

\begin{figure}
	\centering
		\includegraphics[width=.8\textwidth]{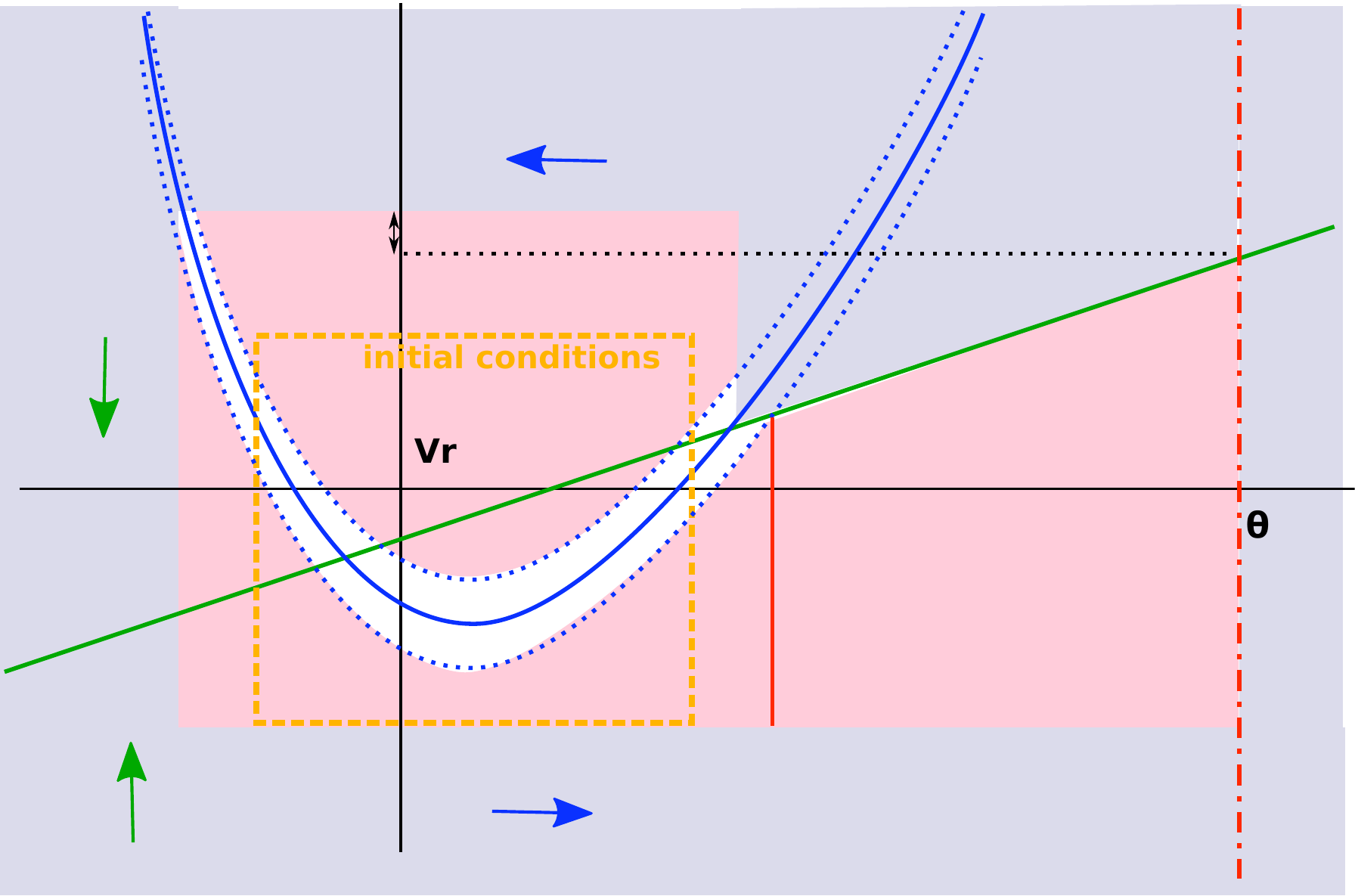}
	\caption{Partition of the phase space. Pink zone: phase-plane integration and white zone: time integration. Orange dotted box: initial conditions.}
	\label{fig:AlgoPartition}
\end{figure}
We are now in position to  to propose a suitable choice of $(M,\,dt,\,dv)$ ensuring a fixed precision $\varepsilon$ on the adaptation variable at the time of the cutoff $v=\theta$. An optimal choice of $(M,\,dt,\,dv)$ would moreover minimizes the number of operations necessary to reach the desired precision $\varepsilon$, but this optimal choice is quite difficult to define. We nevertheless define an even more efficient method in section \ref{sec:AdaptSteps} that involve adaptive integration steps ensuring a fixed absolute error $\varepsilon$. To this purpose, let us assume that we consider a bounded set of initial conditions $[v_1,v_M] \times [w_m,w_M]$ (orange box of figure Fig.~\ref{fig:AlgoPartition}), and let us consider the regions of the phase plane reached by the orbits of the system starting from these initial conditions. 

\begin{lemma}\label{lem:BoundedTrajectories}
	For any initial condition $(v_0,w_0)\in [v_m,v_M]\times [w_m,w_M]$, the thresholded trajectories $(v(t),w(t))$ are contained in a compact set $[v_l, \theta] \times [w_d,w_u]$. Moreover, if $F$ grows faster than $v^{2+\delta}$ for some $\delta>0$, then $v_l$, $w_d$ and $w_u$ are independent of $\theta$, and the time integration phase takes place in a compact set independent of $\theta$.
\end{lemma}
The proof of this quite intuitive lemma is provided in appendix \ref{append:Boundedness}, and is illustrated in Figure~\ref{fig:AlgoPartition}: the zones of the phase plane depicted in gray, and the rest of the phase plane not plotted in the figure constitutes a zones of the phase plane the orbit starting from any initial condition in the subset defined will never reach. 

Now that we restricted possible initial conditions of the system, we know from lemma \ref{lem:BoundedTrajectories} that the values of $(v,w)$ on the orbits are contained in a bounded set. For the quadratic model, these values depend on the cutoff value $\theta$ since the adaptation variable at the times of the spike diverges, and for models satisfying assumption \ref{Assump:convergence}, it can be defined independently of $\theta$. On this bounded set, the functions $F(v)-w+I$ and $a\,(b\,v-w)$ are both bounded, differentiable with bounded derivatives, which allows definition of time steps $dt$ and phase-space step $dv$ that uniformly ensure a given precision of the numerical algorithm. This fact is discussed in appendix \ref{append:Boundedness}, and we now turn to discuss a more efficient method based on the algorithm proposed but involving adaptive integration steps.

This algorithm can be adapted to most classical simulation tools used in computational neuroscience (e.g. BRIAN~\citep{goodman-brette:08} or NEST~\citep{NEST:02}) and that make use of fixed time step methods, by including a condition to switch either to integrating (i) or (ii). 

\subsection{Adaptive integration steps algorithm}\label{sec:AdaptSteps}
Let us consider a constant $M>0$ fixed. In contrast to the fixed integration step methods introduced in the previous section, we define $\varepsilon$ the absolute precision we wants to achieve all along the simulation of the orbits. We also define $DT$ and $DV$ the maximal values of time and phase-space integration steps we want to take into account. An efficient method minimizing the computation operations consists in defining at each operation the optimal integration steps ensuring a precision bounded by $\varepsilon$ at each step of the integration algorithm. The absolute precision is known to be bounded by the second derivative of the vector field multiplied by the integration step. We will therefore define the integration step as a function of the value of the modulus of the derivative of the vector field in each case. Let $(t_0,v_0,w_0)$ be an initial condition of the dynamical system \eqref{eq:GeneralModel} with cutoff and reset, and assume that we iteratively computed the solution up to rank $n$. We therefore dispose of a point $(t_n,v_n,w_n)$. We compute the next point of the orbit $(t_{n+1},v_{n+1},w_{n+1})$ as follows:

\begin{enumerate}
	\item If $ \vert F(v_n)-w_n+I(t_n)\vert < M$, we simulate the solutions of equation \eqref{eq:GeneralModel}. The time step chosen depends on the value of the second derivative of $v(t)$ and $w(t)$, that read:
	\[
	\begin{cases}
		v''(t) = F'(v)(F(v)-w+I) - a(b\,v-w) + I'(t) \\
		w''(t) = a\,(b\,(F(v)-w+I)) -a (b\,v-w)
	\end{cases}
	\]
	where $x'$ denotes the derivative of $x(t)$ with respect to time. The larger time step $dt$ ensuring a precision of order $\varepsilon$ at the point $(v,w)$ therefore reads:
	\[\widetilde{dt}(v,w) =\frac{\varepsilon}{\max(\vert v''(t)\vert , \vert w''(t)\vert) }.\]
	In order to keep enough points in regions where the vector field varies slowly, we define $\tau_n=\min(\widetilde{dt}(v,w), DT)$ and we define:
	\[ \begin{cases}
				v_{n+1} &= v_{n} + \tau_n \, (F(v_n)-w_n+I(t_n)) \\
				w_{n+1} &= w_{n} + \tau_n \, a \, (b\,v_n-w_n)\\
				t_{n+1} &= t_{n} + \tau_n
			\end{cases}
	\]
	\item If $ \vert F(v_n)-w_n+I(t_n)\vert < M$, we advance the simulation by computing the next point through equation \eqref{eq:GeneralModel}. The integration step chosen depends on the value of the second derivative of $T(v)$ and $W(v)$ with respect to $v$ (denoted with a double dot), that reads:
	\[
	\begin{cases}
		\ddot{W}(v) =  \displaystyle{\frac{a\,b}{F(v)-W+I}-\frac{a\,(b\,v-W)}{(F(v)-W+I)^2} -\frac{a\,F'(v)\,(b\,v-W)}{(F(v)-W+I)^2}} \\
		\qquad\qquad \displaystyle{- \frac{a\,(b\,v-W)\,(a\,(b\,v-W)+I'(T))}{(F(v)-W+I)^3}} \\
		\ddot{T}(v) = \displaystyle{-\frac{F'(v)(F(v)-W+I)-a\,(b\,v-W) + I'(T)}{(F(v)-W+I)^3}}
	\end{cases}
	\]
With these expressions, we are in position to choose an optimal integration step depending on the values of the variables $(v,w)$ ensuring the desired precision:
	\[\widetilde{dv}(v,w) = \frac{\varepsilon}{\max(\vert \ddot{W}(v)\vert , \vert \ddot{T}(v)\vert )}.\]
	Here again, in order to keep enough points in regions where the vector field varies slowly, we define $dv_n=\min(\widetilde{dv}(v,w), DV)$ and we define:
	\[ \begin{cases}
				v_{n+1} &= v_{n} + dv_n \\
				w_{n+1} &= w_{n} + dv_n \, \frac{a \, (b\,v_n-w_n)}{F(v_n)-w_n+I(t_n)}\\
				t_{n+1} &= t_{n} + dv_n\,\frac{1}{F(v_n)-w_n+I(t_n)}
			\end{cases}
	\]
\end{enumerate}

The integration steps chosen at each pace of the algorithm ensures that the precision is always bounded by $\varepsilon$. During phase (i), it therefore ensures that we have \[\max (\vert v(t_n)-v_n\vert, \vert w(t_n)-w_n\vert) <\varepsilon\]
and during phase (ii), that \[\max (\vert W(v_n)-w_n\vert, \vert T(v_n)-t_n\vert) <\varepsilon.\]
This algorithm therefore ensures an overall precision bounded by a constant $C$ multiplied by $\varepsilon$, where the constant $C$ is proportional to the maximal derivatives of $v\mapsto W(v)$ and $v\mapsto T(v)$, both of which are uniformly bounded on $Z^*$. 

Note that the choice of the value of $M$ impacts the number of operations performed. Choosing a large value for $M$ implies that the time integration phase is emphasized, whereas a small value of $M$ emphasizes the phase space integration. Note also that the remark on the explosion of $v$ in the algorithm is still valid, and that on a spiking trajectory, the integration step in phase (ii) of the algorithm is soon chosen to be DV.  

\subsection{Numerical Errors accumulation during a spike train simulation}
Similarly to the case treated in section \ref{sec:fixedstep}, we have shown that the evaluation of the spike time and of the value of the adaptation variable at this time yielded errors, but in contrast with fixed time-step method, the absolute precision is controlled in our algorithm. Indeed, let us assume that $(v(t),w(t))$ is a spiking orbit, and that for $t\geq t^*$, the orbit is contained in $Z^*$ and moreover $F(v(t))-w(t)+I(t)\geq M$. Let us denote by $X^*=(t^*,v^*,w^*)$ a point of the orbit in this region of the phase space and $X^{**}=(t^{**},v^{**}, w^{**})$ another such initial condition. We are interested in the divergence of the trajectories starting from $X^*$ and $X^{**}$ and solutions of equations \eqref{eq:WT}, denoted respectively $(T_1(v),W_1(v))$ and $(T_2(v),W_2(v))$. Integrating formally the equation and subtracting the expressions obtained, we get, denoting $G_1=(F(v)-W_1(v)+I(T_1(v)))$ and $G_2=(F(v)-W_2(v)+I(T_2(v)))$
\[
\begin{cases}
	T_1(v)-T_2(v) &= (t^*-t^{**})+\int_{v^*}^{v^{**}} \frac{1}{G_1(u)}\,du + \int_{v^{**}}^v \frac{1}{G_1(u)}-\frac{1}{G_2(u)}\,du\\
	W_1(v)-W_2(v) &= (w^*-w^{**})+\int_{v^*}^{v^{**}} \frac{a\,(b\,u-W_1(u))}{G_1(u)}\,du + \int_{v^{**}}^v \frac{a\,(b\,u-W_1(u))}{G_1(u)}-\frac{a\,(b\,u-W_2(u))}{G_2(u)}\,du\\
\end{cases}
\]
These equations are contracting and will result in an exponential decrease of the error made. Let us demonstrate this fact by first considering the case where the input current $I$ is constant (the general case can be done along the same lines). Making more explicit the above expressions, we get:

\[T_1(v)-T_2(v) = (t^*-t^{**})+\int_{v^*}^{v^{**}} \frac{1}{G_1(u)}\,du + \int_{v^{**}}^v \frac{W_1(u)-W_2(u)}{G_1(u)G_2(u)}\,du\]
and therefore the error made on the evaluation of the spike time is governed by the error made on the evaluation of the adaptation variable, which reads:
\[W_1(v)-W_2(v) = (w^*-w^{**})+\int_{v^*}^{v^{**}} \frac{a\,(b\,u-W_1(u))}{G_1(u)}\,du + \int_{v^{**}}^v \frac{a\,(b\,u-(F(u)+I))}{G_1(u)G_2(u)}\,(W_1-W_2)\,du\]
Let $K>0$ fixed. There exists $L>v^{**}$ such that for all $u\geq L$ we have $(b\,u-(F(u)+I))<0$. Therefore, as $v$ increases, the equation on the distance between $W_1$ and $W_2$ becomes contracting and the distance between the two adaptation variable trajectories is reduced. This result extends to the time variables, and hence in contrast to the system \eqref{eq:GeneralModel}, the error is here reduced instead of exponentially increased, and the larger the cutoff $\theta$ is, the smaller the distance between the two computed trajectories\footnote{The complete proof of this result involves Gronwall theorem, and can be related to the contraction theory (see e.g.\cite{slotine-li:91,lohmiller-slotine:98})}. 

After the first spike is fired, the numerically computed spike time $t^{**}$ and the value of the adaptation variable at this time $w^{**}$ differ from the actual spike time $t^*$ and value of the adaptation variable $w^*$, and the difference is bounded by the precision $\varepsilon$ chosen. We have noticed that the equation \eqref{eq:GeneralModel} was expanding along spiking trajectories, i.e. that the distance between two orbits starting from different initial conditions diverge close to a spike emission. However, in the case of the algorithm we propose here, instead of simulating equations \eqref{eq:GeneralModel} close to the spike emission, we rather simulate equations \eqref{eq:WT}. We showed that these equations, contrarily to the original dynamical system, are contracting when $v$ is large enough: two orbits starting from different initial conditions converge exponentially the one towards the other. Therefore, instead of increasing the errors, the simulation of \eqref{eq:WT} will reduce the error made on the evaluation of the spike time and on the adaptation variable at this time. Though the discontinuity might produce large errors at the reset points, the eventually computed spike time and adaptation variable are reliable because of the contracting properties of the flow of equations \eqref{eq:WT}.

\subsection{More refined integration procedures}
In the algorithm we proposed in this section, we made use of a Euler scheme, either with fixed or adaptive integration steps. The exact same procedure can be performed with more refined numerical integration procedures, such as Runge-Kutta methods, Adams-Moulton or Adams-Bashforth multistep methods, or implicit schemes, in each phase of the algorithm, namely the phase (i) of time integration and the phase (ii) of phase-space integration. These methods will therefore present both the advantages of finer integration methods and of our integration scheme that chooses between the two description of the orbits \eqref{eq:GeneralModel} and \eqref{eq:WT} the one that minimizes numerical errors, and also the fact that the errors made on the evaluation of the spike time and the adaptation variable at this time exponentially decreases in the integration phase (ii). These therefore provide a set of efficient simulation methods. 

\section{Discussion}\label{sec:Results}
In this paper we studied the error produced by fixed time step integration methods for nonlinear bidimensional neuron models and showed that the error produced, even if these are of order $1$ in the time step chosen, are proportional to a function of the parameters and of the cutoff chosen to identify the spikes, and this function is unbounded as the cutoff is increased. This fact implies that in order to achieve a fixed precision on the global algorithm and accurately simulate the explosion time, one needs to use a large enough threshold value $\theta$ and a very small time step in order to accurately estimate the spike time and the value of the adaptation variable at this time. We saw that the main difficulty arises close from the spike time, since the explosion of the membrane potential variable produces large errors and the differential of the adaptation variable also diverges at this time. The small time step needed to accurately evaluate the crucial values of the spike time and the adaptation variable at these time will result in an excessive precision in regions where the vector field is smooth and an increased computational time.

In order to overcome this difficulty, we proposed an alternative algorithm that is based on the numerical computation of the orbits in the phase plane in regions where the modulus of the vector field is large enough. This algorithm, similar to classical variable integration step methods, allows to choose optimal steps in order achieve the desired precision in a minimal number of operations, which can be uniformly bounded as shown in section \ref{sec:MyMethod}. Therefore this method provides us with a stable method that allows considering large thresholds and that precisely evaluates the spike times and the value of the adaptation variable at these times in a reduced number of operations. 

In this section we specifically focus on the impact of numerical errors in the evaluated trajectories and on the computational efficiency of the algorithms. 

\subsection{Quantitative imprecision yield dramatic qualitative errors}
We now discuss the importance of numerical errors on the simulations from a computational neuroscience viewpoint. The models we addressed in this paper are particularly used to simulate spike times in order to identify different spike patterns, such as regular spiking, bursting and chaotic spiking (see e.g. \citep{izhikevich:03, izhikevich:07, brette-gerstner:05,touboul:08, touboul-brette:09}). We showed that the spike times and the adaptation variable at these times can be substantially modified when using fixed time step methods, and this effect can produce qualitative distinctions between firing patterns. Indeed, it was shown in \citep{touboul-brette:09} that the spike pattern depended sensitively on the value of the adaptation variable at the times of the spikes and that these spike patterns undergoes bifurcations. Since we have seen that fixed time-step Euler scheme produces a systematic negative error (i.e. the evaluated spike time was strictly inferior its actual value as a solution of the continuous differential equations), this error can make the system cross bifurcation points and change spike patterns. 

We illustrate this fact for instance on the quadratic model with original parameters provided in \citep{izhikevich:03}: in this case, $F(x)=0.04\,x^2+5\,x+140$ and $c=-59.9$, parameters that were chosen to fit intracellular recordings. By choosing the parameters $a=0.02$, $b=0.19$, $d= 1.15$, $\theta=30$ and a constant input $I=7.6$, the system produces bursts with two spikes per bursts. We simulate for $T=0$ to $1000$ in order to record a long enough spike train, and plot the sequence of values of the adaptation variable at the time of the spikes. These values allow discriminating between bursts (that produce periodic sequences of reset values with period two) and regular spiking (that correspond to a fixed point of the sequence). Results of the simulation are plotted in Figure~\ref{fig:Boulette}.
\begin{figure}
	\begin{center}
		\includegraphics[height=.65\textheight]{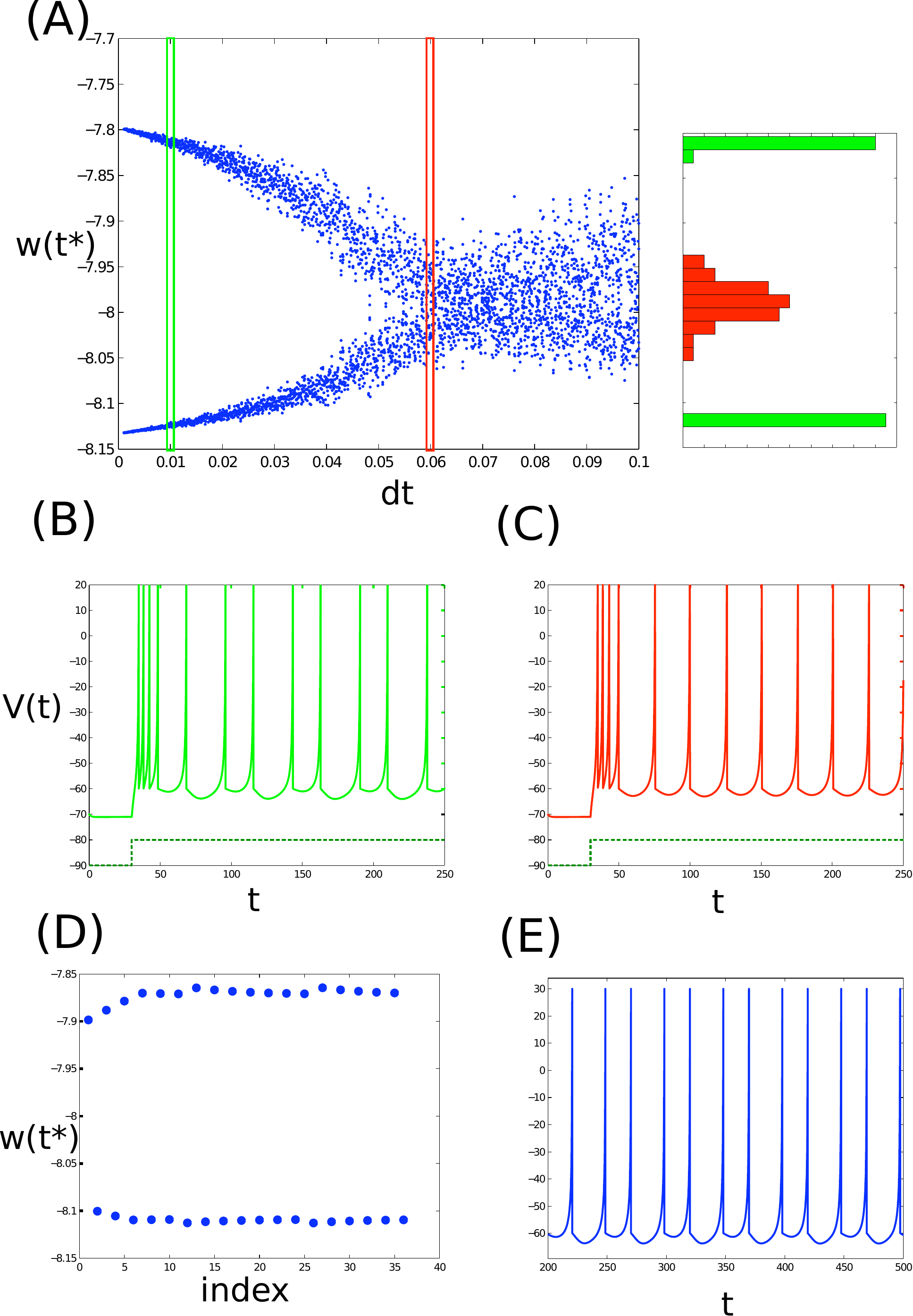}
	\end{center}
	\caption{Simulation of the quadratic model with parameters given in the text. (A) represents the value of the adaptation variable at the spike times as a function of a fixed time-step $dt$ for an Euler scheme. (left) We observe that the bursting nature of the trajectory is lost and one recovers a regular spiking behavior if the time step is not small enough. (right) represents the histogram of these values for $dt=0.01$ (green) and $dt=0.1$ (red). (B) and (C) correspond to sample trajectories computed for $dt=0.01$ and $dt=0.1$ respectively (stopped at $t=250$ for legibility). (D) represents the sequence of reset values for the variable integration step algorithm with precision set at $\varepsilon=0.1$ (similar to (A) but for a fixed integration scheme and where the abcissa corresponds to the number of spike fired before this value is obtained) and (E) a sample trajectory simulated by this algorithm. }
	\label{fig:Boulette}
\end{figure}
We observe that for $dt<0.025$ the bursting nature of the trajectory is recovered, but for any time step greater than this value, the bursting nature of the spike pattern produced is not clearly identifiable. Moreover, we observe that for $dt>0.05$, the spike patterns produce corresponds to a regular spiking behavior: the system crossed the period-doubling bifurcation because of numerical errors, and this is not a purely visual effect: the sequence of reset values is unimodal (as shown in Figure~\ref{fig:Boulette}(A)). Therefore, on our simulation interval, one needs to perform at least $30\,000$ operations to recover the nature of the desired spike train, and at his resolution the burst is quite perturbed by numerical imprecision. 

For the same parameters, we set a precision $\varepsilon=0.01$ and compute the trajectory using our fixed precision algorithm. The algorithm performs $2000$ operations and achieves a great precision on the value of the adaptation variable at the spike time, that is not comparable to what is obtained by the fixed time-step Euler method. Actually the precision of the algorithm is smaller than the value of $\varepsilon$ chosen because of the control put on the value of the integration steps in order for these not to exceed certain values in regions where the vector field varies slowly. For such precision, one would necessitate to use $dt=0.01$ corresponding actually to $100\,000$ operations. Therefore, the computational gain of using the new algorithm instead of a fixed time step Euler method ranges from $30$ to $50$.

We now compare our algorithm to different classical algorithms with adaptive time-steps encoded in Matlab. All classical algorithms are tested, and both simulated trajectory and execution time are compared to our algorithm. For small time intervals, all methods fairly well approximate the solution, but after some time, the spikes are shifted and though all methods still conserve the bursting nature of the trajectory, all the solutions are increasingly shifted. We set an absolute precision of $0.01$, and compare the computation times, summarized in table \ref{tab:MatlabComparison}, and the trajectories. We observe that the algorithm proposed in the manuscript is way faster than any other method, which was expected since it is adapted to the very nature of the problem and is not as generic as Matlab's routines. We observe that all the simulations increasingly large delays in the spikes compared to our simulation algorithm (see Figure~\ref{fig:Delays}), which is probably linked with the fact, evidenced here in fixed time-step Euler integration scheme, that estimated spike times with fixed time step integration schemes or adaptive time step algorithms present a systematic positive delay in the estimation of the spike time, and also that the errors accumulate all along the emission of a spike train. However, an adaptive method seems to stand out: the {\tt ode113} routine, implementing a variable order Adams-Bashforth-Moulton PECE solver, a linear multistep integration method (see e.g. \citep{hairer-norsett:93}) provides a very good match with our simulation algorithm, with reasonable computational time. None of the other Matlab's algorithms reliably simulate the system, even if the precision is set to $\varepsilon=0.01$, on the interval $[0,\,1000]$. This will also be the case of the fixed time-step Euler scheme: whatever the time step $dt$ is, small variations will progressively shift the spikes. Besides this imprecision, all algorithms present moreover longer execution time evaluated on a Macbook with processor 2 GHz Intel Core 2 Duo and memory  2GB  1067 MHz DDR3. However, though these differences, it is important to note that all the algorithms reproduce the bursting property and provide a good match with the sequence of reset values expected. This is not the case of the Euler method for time step $dt>0.03$. 
\begin{table}
	\begin{center}
	\begin{tabular}{|c|c|c|c|c|c|c|c|c|}
		\hline
		Algorithm & Paper's & {\tt ode23} & {\tt ode45} & {\tt ode113} & {\tt ode15s} & {\tt ode23s} & {\tt ode23t} & {\tt ode23tb}\\
		\hline
		Exec. time & 0.0053 & 0.7478 & 0.6730 & 0.6567 & 1.993 & 2.301 & 1.7538 & 1.7756\\
		\hline
	\end{tabular}
	\end{center}
	\caption{Compares the execution time of the paper's algorithm with different algorithms proposed in Matlab. }
	\label{tab:MatlabComparison}
\end{table}
This conclusion extends to most adaptive time-step integration methods, such as the one implemented in NEST~\citep{NEST:02}.
\begin{figure}
	\centering
		\subfigure[All Algorithms]{\includegraphics[width=.4\textwidth]{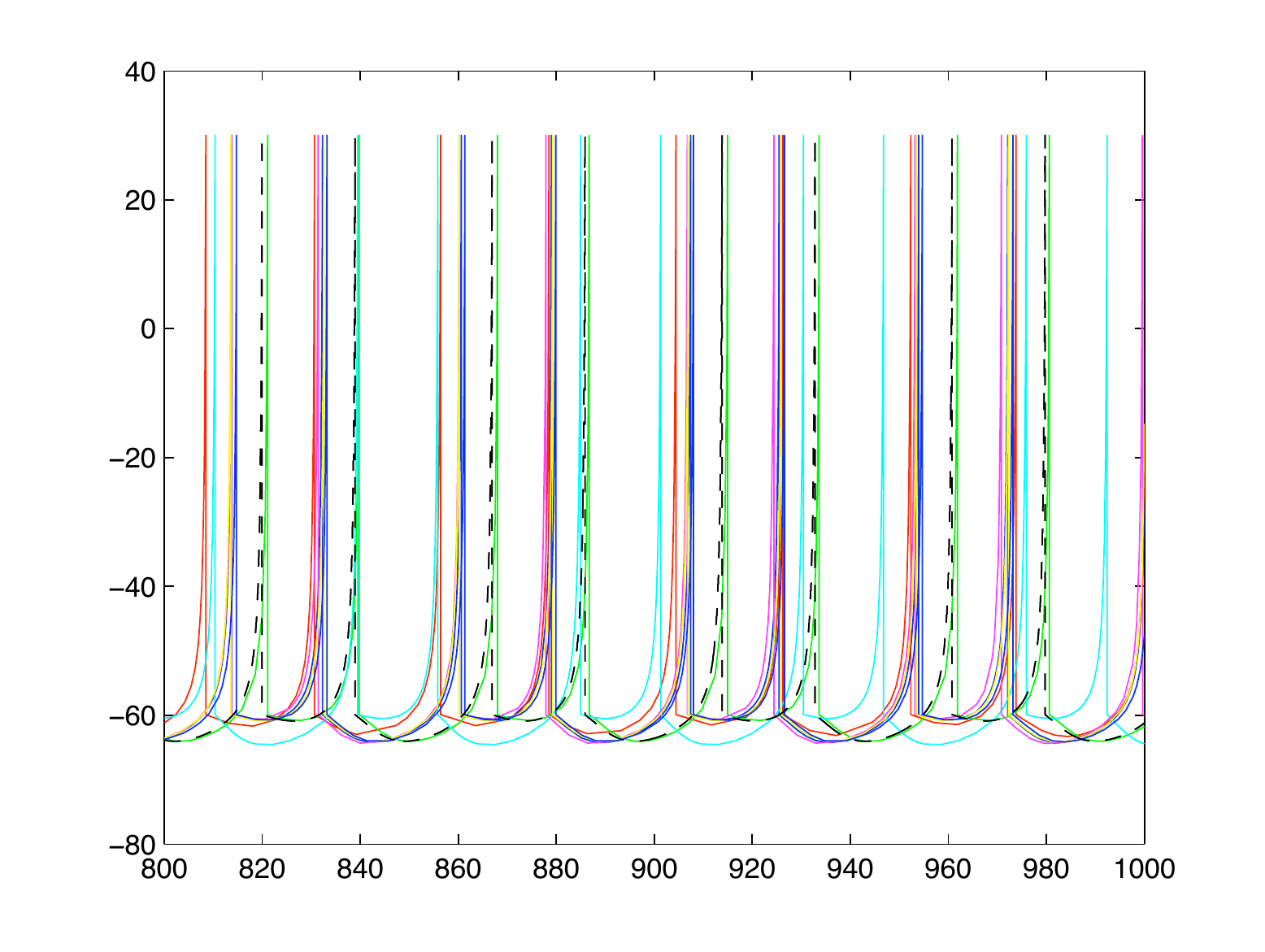}}\qquad
		\subfigure[{\tt ode113}]{\includegraphics[width=.4\textwidth]{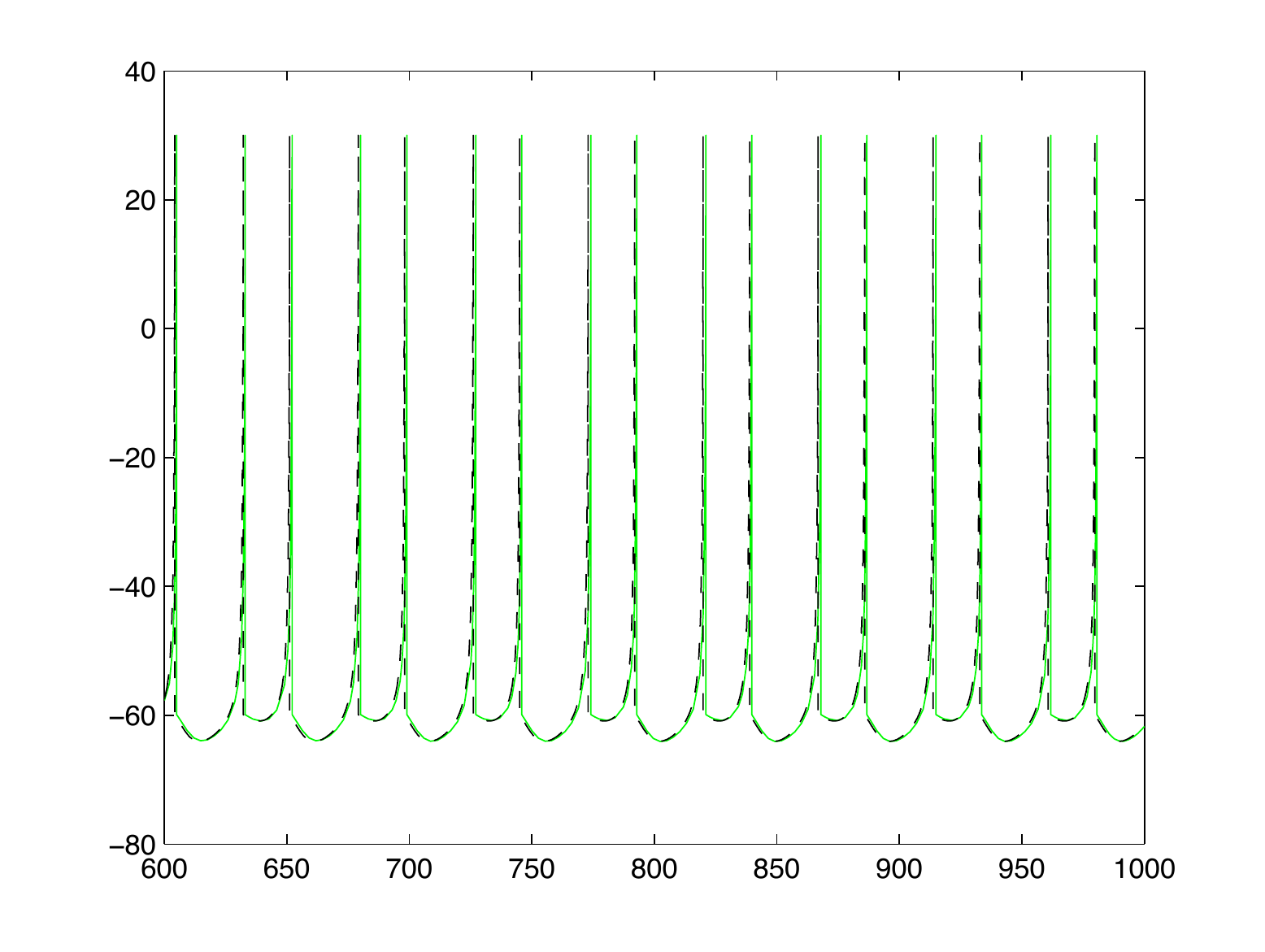}}
	\caption{Comparison of the solutions computed with different integration methods implemented in Matlab. The variable order Adams-Bashforth-Moulton PECE solver ({\tt ode113}) provides a very good match with the algorithm proposed in this paper.}
	\label{fig:Delays}
\end{figure}

\subsection{Number of operations needed to achieve fixed precision for fixed time-step Euler scheme}
In this section we compare the computational efficiency of the two proposed algorithms. In Figure~\ref{fig:Errors}, we plotted the error made on the evaluation of the adaptation variable as a function of the cutoff $\theta$ chosen. Let us denote by $E(\theta)$ this function. The error is then proportional to $E(\theta)\,dt$ and therefore these curves are directly related to the time step needed to achieve a given precision. Indeed, if one wants to get an absolute precision bounded by $\varepsilon$ on the adaptation variable, he needs to choose $dt < \varepsilon / \vert E(\theta) \vert$ which implies that the number of operations needed is proportional to $\vert E(\theta)\vert / \varepsilon$, which diverges as the cutoff is increased. The curves on the error $\vert A(v)\vert $ and $\vert B(v)\vert$ are therefore proportional to the number of operations needed to achieved fixed precision. 

\subsection{Further Topics}
In models satisfying assumption \ref{Assump:convergence} such as the quartic and the exponential adaptive models, the adaptation variable at the times of the spike converges, and therefore it can be very convenient to use large values of the cutoff in order to get rid of the artificial  strictly voltage threshold similar to classical integrate-and-fire models \citep{lapicque:07,stein:67} by a more realistic spike initiation (see e.g. \citep{latham-richmond-etal:00,fourcaud-trocme-hansel-etal:03}) and fully taking advantage of this property of nonlinear integrate-and-fire neurons. The algorithm proposed here allows considering larger values of cutoff $\theta$ without a prohibitive increasing the computational time, which allows more accurately simulating the blowing-up system. Another interest of the proposed numerical method is the flexibility in terms of spike times: in contrast to all fixed time-step methods where the spike times are aligned on a grid, our algorithm can have arbitrary spike times, which is an important property from the biological viewpoint and for the computation of large networks or over large period of times.

The algorithm we proposed here can easily be used for simulations of networks composed of neurons of the type studied here. This new method will therefore be very useful for the simulations of very large networks using such models by providing a more precise and computationally more efficient method to simulate the solutions of these equations. 

\paragraph{Acknowledgement} This work was partially supported by ERC grant NERVI-227747. The author warmly acknowledges Romain Veltz and Romain Brette for useful discussions and suggestions. 

\appendix

\section{Some Theoretical Results}\label{sec:theo}
In this appendix we review and prove some properties of the dynamical system \eqref{eq:GeneralModel} of interest in the main text.

We first recall that this dynamical system either has zero, one or two fixed points depending on the values of the parameters. More precisely, let us denote by $m(b)$ the unique minimun of the convex function $v \mapsto F(v) -b\,v$, and by $v^*(b)$ the unique point where this minimum is reached (i.e. the unique solution of $F'(v^*(b))=b$, which exists under assumption \ref{Assump:Blow}). It is proved in \citep{touboul:08} (see figure \ref{fig:fixedPoints}) that:
\begin{itemize}
	\item if $I>-m(b)$ then the system has no fixed point;
	\item if $I=-m(b)$ then the system has a unique fixed point, $(v^*(b), w^*(b))$, which is non-hyperbolic. It is unstable if $b>a$. 
	\item if $I<-m(b)$ then the dynamical system has two fixed points $(v_-(I,b), w_-=b\,v_-(I,b))$ and $(v_+(I,b), w_+=b\,v_+(I,b))$ such that \[v_-(I,b)<v^*(b)<v_+(I,b).\]
	 \noindent The fixed point $v_+(I,b)$ is a saddle fixed point, and the stability of the fixed point $v_-(I,b)$ depends on $I$ and on the sign of $(b-a)$: 
	 \begin{enumerate}
	  \item If $b<a$ then the fixed point $v_-(I,b)$ is attractive. 
	  \item If $b>a$, there is a unique smooth curve $I_s(a,b)$ defined by the implicit equation $F'(v_-(I_s(a,b),b)) = a$. This curve is given by the equation $I_s(a,b) = b v^*(a) - F(v^*(a))$ where $v^*(a)$ is the unique solution of  $F'(v^*(a))=a$. 
	 \begin{enumerate}
	  \item If $I<I_s(a,b)$ the fixed point is attractive. 
	  \item If $I>I_s(a,b)$ the fixed point is repulsive.
	\end{enumerate}
\end{enumerate}
\end{itemize}

\medskip
In \citep{touboul-brette:09}, the authors thoroughly define a Markov partition of the dynamics in each case, that involve a \emph{spiking zone} which is stable under the flow and where any orbit with initial condition in the zone blows up in finite time (and fires a spike). The precise description of this zone is quite complex and involves the description of the stable manifold of the saddle fixed point $v_+(I,b)$, which does not have a simple analytical expression. We are interested here in defining a sub-zone of the spiking zone stable under the flow of the differential equation that allows a simple computation of the spike time and of the adaptation variable at this time. In definition \ref{def:Z}, we introduced a subset of the phase plane $Z^*$. This subsets of the phase plane present the interest to satisfy the following properties
		\begin{figure}
			\centering
				\includegraphics[width=\textwidth]{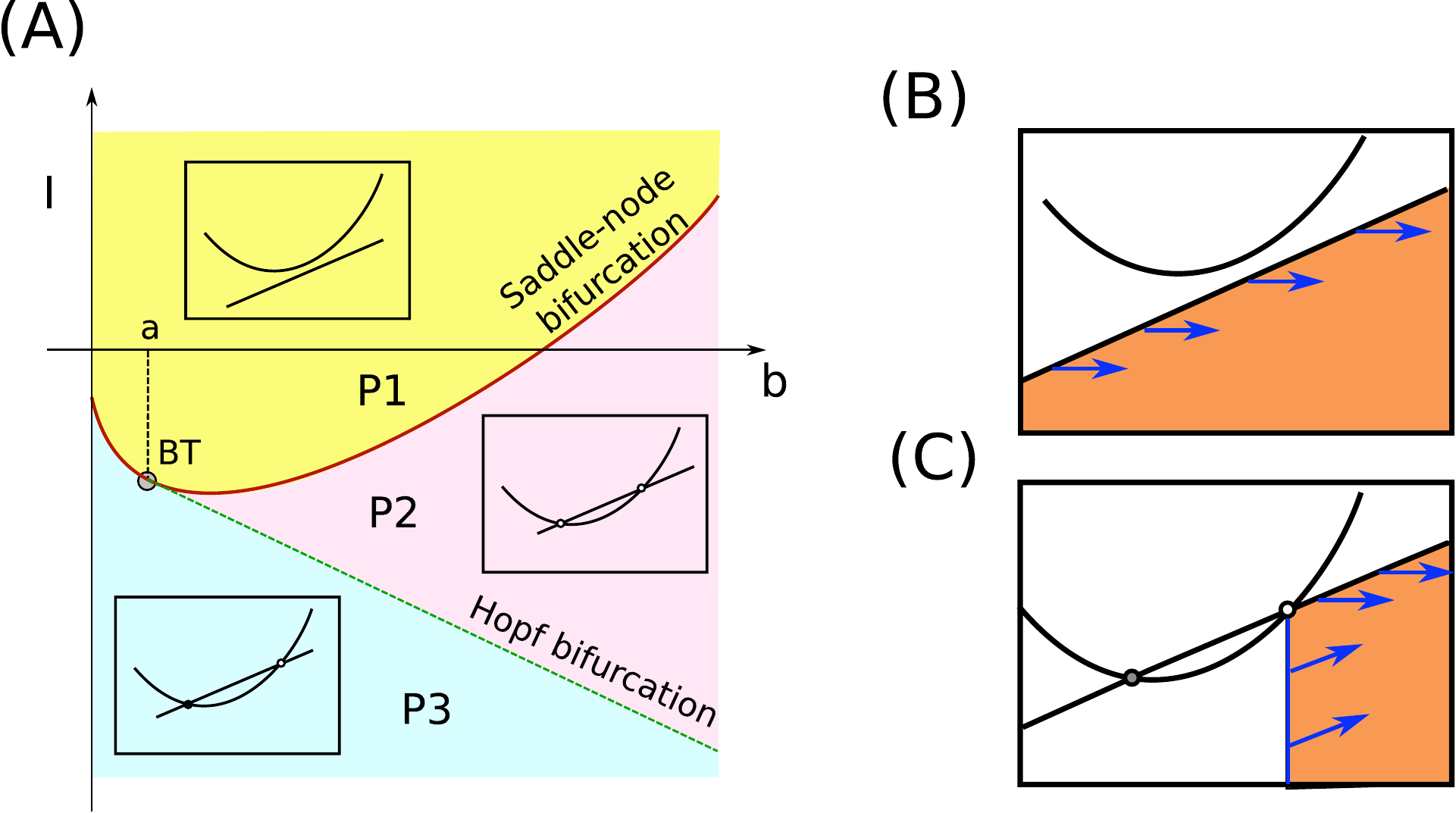}
			\caption{(A) Fixed points and stability in the parameter plane $(I,b)$. The parameter zone $P1$ corresponds to a case with no fixed points, and the related zone $Z$ is displayed in (B). In the parameter zones $P2$ and $P3$, the $Z$ zone is displayed in (C). }
			\label{fig:fixedPoints}
		\end{figure}

\begin{theorem}
	Assume that the input current $I(t)$ depends on time, and moreover that $t\mapsto I(t)$ is lower bounded (at least on the time interval considered $(\tau_0,\tau_1)$), i.e. $I(t) \geq I^*$ for any $t\in (\tau_0,\tau_1)$. Then we have:
	\renewcommand{\theenumi}{(\roman{enumi})}
	\begin{enumerate}
		\item $Z(I^*,b)=:Z^*$ is stable under the flow of the equation
		\item Let $(v_0,w_0) \in Z^*$, and denote $(v(t),w(t))$ the solution of equations \eqref{eq:GeneralModel} having initial condition $(v_0,w_0)$ at time $t_0 \in [\tau_0,\tau_1]$. Then $t \mapsto v(t)$ is strictly increasing, therefore invertible. Its inverse function $T(v)$ is differentiable, hence so is $W(v)=w(T(v))$, and these variable  satisfy the nonlinear differential equation:
		\begin{equation}\label{eq:WT2}
			\begin{cases}
				\displaystyle{\der{T}{v}} &= \displaystyle{\frac{1}{F(v)-W+I(T)}}\\
				\\
				\displaystyle{\der{W}{v}} &= \displaystyle{\frac{a\,(b\,v-W)}{F(v)-W+I(T)}}
			\end{cases}
		\end{equation}
			The explosion time (resp. the value of the adaptation variable at this time) is the limit of $T(v)$ (resp. $W(v)$) when $v\to \infty$.
	\end{enumerate}
\end{theorem}

\begin{proof}
	\renewcommand{\theenumi}{(\roman{enumi})}
	\begin{enumerate}
		\item The stability of $Z^*=Z(I^*,b)$ under the flow of the equation stems from the fact that the vector field points inwards in $Z^*$ on its boundary (see figure \ref{fig:fixedPoints}). Let us first address the case $I^*>-m(b)$. In that case, we have (by definition of $m(b)$) $F(v)-b\,v+ I(t) >0$ for all $v\in \R$ and all $t \in (\tau_0,\tau_1)$. $Z^*$ is defined as  $\{(v,w) \; ;\; \Gamma(v,w) <0\}$ where $\Gamma(v,w)=w-b\,v$. We use a proof by contradiction. Let us assume that the zone $Z^*$ is not stable under the flow, and that a solution $(v(t),w(t))$ with initial condition $(v_0,w_0) \in Z^*$ at $t_0 \in (\tau_0,\tau_1)$ exits $Z^*$ at time $t_1$. This time $t_1$ is defined as the first exit time of $Z^*$, i.e.
		\[t_1 = \inf\left\{t\geq t_0 ; \; (v(t),w(t))\notin Z^*\right\}.\]
Because of the continuity of the boundary and of the solutions, necessarily we have $w(t_1)=b\,v(t_1)$. Moreover, at this point we have:
		\begin{align}
			\nonumber \left .\der{\Gamma(v(t),w(t))}{t}\right \vert_{t=t_1} &= \left . a\,(b\,v(t) - w(t)) - b\,(F(v(t))-w(t)+I(t))\right \vert_{t=t_1}\\
			\nonumber &= -b \, (F(v(t_1))-b\,v(t_1) + I(t_1))\\
			\label{eq:proved} & \leq -b \, (F(v(t_1))-b\,v(t_1) + I^*)
		\end{align}
		This last expression is strictly negative because of the fact that $I^*>-m(b)$. Therefore, since $\gamma:t\mapsto\Gamma(v(t),w(t))$ is a differentiable function of time such that $\gamma(t_1)=0$ and $\der{\gamma}{t}\vert_{t=t_1}<0$, there exists $\eta>0$ such that $\gamma(t)\leq 0$ on $[t_1,t_1+\eta)$, hence $(v(t),w(t))\in Z^*$ on $[t_1,t_1+\eta]$ which contradicts the definition of time $t_1$ as the exit time from $Z^*$. 
		
		In the case where $I^*\leq -m(b)$, the zone $Z^*$ is defined by the additional condition that $v\geq v_+(I^*,b)$. On this zone, $v\mapsto F(v)-b\,v$ is non-negative and increasing because of its convexity property. Therefore, on $Z^*$, we have $(F(v(t_1))-b\,v(t_1) + I^*)>0$. Let us consider $(v(t),w(t))$ a solution with initial condition $(v_0,w_0) \in Z^*$. We have $v_0 \geq v_+(I^*,b)$. While $(v(t),w(t))$ is in $Z^*$, $v(t)$ is strictly increasing, and hence the solution will never cross the boundary $v=v_+(I^*,b)$, and can only exit $Z^*$ through the boundary $w=b\,v$. On this boundary, the same argument as in the case where $I^*>-m(b)$ applies, which ends the proof.
		
		Note that when the current $I(t)$ is constant, this result directly comes from the property that the vector field points inwards the zone $Z^*$ (and inwards any $Z(V)$ for $V>v_+(I,b)$, with the convention $v_+(I,b) = -\infty$ if $I>-m(b)$), making of the zone $Z^*$ a \emph{trapping region} as referred in \cite[Chap. 7.3.]{strogatz:94} (a proof of this property can be found for instance in \citep{viterbo:05}).
		
		\item  Let $(v_0,w_0) \in Z^*$. Then necessarily $(v_0,w_0) \in Z(V)$ for some $V$ such that $v_+(I,b)<V<v_0$. In $Z(V)$, we have $F(v)-w+I \geq F(v) - b\,v +I > 0 $, therefore $t\mapsto v(t)$ is a strictly increasing function. Moreover, since the derivative of $v(t)$ never vanishes, $t\mapsto v(t)$ is invertible with differentiable inverse function $T(v)$ by application of the inverse function theorem. 
		
		From the stability property of $Z^*$, for all $t\geq t_0$ we have $w\leq b\,v$ and hence
		\[\der{v(t)}{t} = F(v)- w+ I(t) \geq F(v) -b\,v + I^*\]
		whose solution blows up in finite time under assumption \ref{Assump:Blow}, by the virtue of Gronwall's lemma, see \citep{touboul:09}. 
		
We denoted by $T(v)$ the inverse of the function $t\mapsto v(t)$ and noticed that it was differentiable, and therefore $W(v)$ the composed application $W(v)=w(T(v))$ is also differentiable. It is clear that the explosion time corresponds to the limit of $T(v)$ when $v\to \infty$ and the value of the adaptation variable at the explosion time the limit of $W(v)$ when $v\to\infty$. We have by definition $v(T(v)) = v$, and differentiating by $v$ we get:
		\begin{align*}
			\der{v(T(v))}{v} &= \left . \der{v(t)}{t}\right \vert_{t=T(v)}\der{T(v)}{v}\\
			& = \left (F\big(v(T(v))\big) - w(T(v)) + I(T(v)) \right) \der{T(v)}{v}
		\end{align*}
		and it is also equal to one since $v(T(v))=v$. Therefore we have, since $F(v) - w + I(t)$ never vanishes on $Z(V)$:
		\[\der{T(v)}{v} = \frac{1}{F(v) - W(v) + I(T)},\]
		and using the differentiation formula of a composed function,
		\[\der{W}{v} = \frac{a\,(b\,v-W)}{F(v)-W + I(T)}.\]
	\end{enumerate}
\end{proof}


Another property we can deduce from \citep{touboul:09} is that the adaptation variable in a left neighborhood of the explosion time is always negligible compared to $v$.
This remark implies that the differential of $w(t)$ tends to infinity at the explosion times of $v$. However, theorem \ref{theo:blowup} shows that in the region where the neuron elicits a spike, the time of the spike and the adaptation variable satisfy a well-behaved differential equation \eqref{eq:WT2}, fact on which our simulation algorithm of section \ref{sec:MyMethod} is grounded.


\section{Trajectories starting from a given set of initial conditions are bounded}\label{append:Boundedness}
In this appendix we demonstrate lemma \ref{lem:BoundedTrajectories}, which is recalled here:

\begin{lemma}
	For any initial condition $(v_0,w_0)\in [v_m,v_M]\times [w_m,w_M]$, the thresholded trajectories $(v(t),w(t))$ are contained in a compact set $[v_l, \theta] \times [w_d,w_u]$. Moreover, if $F$ grows faster than $v^{2+\delta}$ for some $\delta>0$, then $v_l$, $w_d$ and $w_u$ are independent of $\theta$, and the time integration phase takes place in a compact set independent of $\theta$ (see Figure~\ref{fig:AlgoPartition}: the zones of the phase plane depicted in gray, and the rest of the phase plane not plotted in the figure constitutes a zones of the phase plane the orbit starting from any initial condition in the subset defined will never reach). 
\end{lemma}

\begin{proof}
	Let $(v_0,w_0)\in [v_m,v_M]\times [w_m,w_M]$ and consider the orbit of the dynamical system \eqref{eq:GeneralModel} with $(v_0,w_0)$ an initial condition in this region. We start by considering $I(t)$ constant. The trajectories for $I(t)$ non constant will be bounded by the values the orbit can reach for the vector associated with the minimal and maximal values of the $I(t)$ because of the monotony of the vector field and because of Gronwall's inequality.
	
	First of all, the case where $F(v)-b\,v+I >0$ for all $v$ corresponding to the case where the nullclines never intersect is simple to treat. Indeed, in that case the spiking zone is simply $\{(v,w); w\leq b\,v\}$ and it is a trapping zone, i.e. any trajectory with initial condition in the spiking zone will never escape from it. For any initial condition in this spiking zone (zone I in Figure~\ref{fig:AlgoPartition}(A)), $w_0\leq w(t)\leq b\,v(t) $ and $v_0 \leq v(t)\leq \theta$ and hence $w(t)\in [w_m,\max(b\theta,\,w_M)]$  and $v\in [v_m,\,\theta]$. If $(v_0,w_0)$ in zone II of Figure~\ref{fig:AlgoPartition}(A), i.e. if $bv_0\leq w_0\leq F(v_0)+I$, the variable $v(t)$ is increasing all along the trajectory and therefore $\theta \geq v(t)\geq v_0\geq v_m$. The variable $w(t)$ is initially decreasing (therefore bounded by $w_0\leq w_M$) before reaching the $w$-nullcline and then increases in the spiking zone and where trajectories are bounded by $b\,\theta$. Hence $b\,v_m \leq b\,v_0 \leq w(t)\leq \max(w_M, b\,\theta)$. Eventually in zone III of Figure \ref{fig:AlgoPartition}(A), i.e. $w_0\geq F(v_0)+I$, both the value of $v(t)$ and of $w(t)$ are initially decreasing and will reach in finite time the $v$-nullcline where it enters zone II. Before reaching this zone, the minimal value reached by the variable $v(t)$ is greater than $v_1=\min(F^{-1}(w_0-I))$ and smaller than $v_0\leq v_M$, and the value of $w(t)$ is smaller than $w_0\leq w_M$ and greater than $\min(F(v)+I)$ which is reached at a point denoted $v^*$. From these initial conditions in zone II the same analysis as previously done extends the boundaries to $\theta \geq v(t)\geq \min(v_m,v_1)$ and $b\,\min(v_1,v_m,v^*) \leq w(t)\leq \max(w_M, b\,\theta)$. 
	
	After resetting, the initial condition is on the line $c$ with values of $w$ in the interval $[\min(v_1,v_m,v^*)+d,\max(w_M, b\,\theta)+d]$ and a similar analysis provides us with the boundaries
	\[(v(t),w(t)) \in [\min(v_m,v_1^*)] \times [b\,\min(v_1^*,v_m,v^*,c), \max(w_M, b\,\theta)] \]
	with $v_1^*=\min(F^{-1}(\max(w_M, b\,\theta)+d-I))$.
	
	Moreover, in the case where $F$ grows faster than $v^{2+\delta}$ for some $\delta>0$, it was proved in \citep{touboul-brette:09} by a thorough analysis of the spike and reset process that the maximal value reached by the variable $w$ in the spiking zone is bounded by a value $\Phi^*$ independently of the cutoff, making the boundaries provided independent of $\theta$.

In the case where the nullclines do meet, we perform a similar analysis. 
First of all, it is clear that the value of $w(t)$ before resetting is bounded by $\max(w_M,b\,\theta)$. After the reset, $w$ is therefore upperbounded by $w_u=\max(w_M,b\,\theta)+d$. The value of $v$ is cut at $\theta$ and hence any thresholded trajectory have $v(t)\leq theta$. We assume that $F(c)-w_u+I^*<0$. At this point we therefore have $v$ decreasing, and the minimal value of $v$ after a reset corresponds to this value and is lowerbounded by the solution of $F(v)-w_u+I^*=0$ i.e. $v\geq \min(F^{-1}(w_u-I^*))$. If $w_u<F(c)+I^*$, then the value of $v$ on any reseted trajectory is greater than $c$. Therefore, we conclude that $v\geq v_l=\min(c,v_m,\min(F^{-1}(w_u-I^*)))$. Eventually, this value provides us with a minimal value of $w$ along the trajectories. Indeed, the minimal value of $v$ after reset implies that any reseted trajectory will have $w(t)\geq b\,v_l$. If $w_0$ is such that $F(v_0)-w_0+I^*<0$, then in this region 
	
First of all, because of the structure of the flow and the reset condition, the maximal value reached on the reset line $v=c$ is upperbounded by $b\,\theta+d$. This bound can be made independent of $\theta$ if the function $F$ grows faster than $v^{2+\delta}$ for some $\delta>0$. In that case, there exists a value denoted $\Phi^*+d$, as defined in \citep{touboul-brette:09}, and corresponding to the value where the adaptation variable will be reset when starting just below $w^*=F(c)+I^*$. This value does not depends on theta and can be evaluated previously before simulations, by simulating equation \eqref{eq:WT2} with initial condition $(c,w_0)$ with $w_0$ slightly below $w^*=m(b)$ the minimum of the function $F(v)-b\,v+I$. Note that if $\theta$ is not chosen very large, or if one does not want to vary $\theta$, this value $\Phi(w^*)$ can be for instance replaced by the an upper-bound $b\,\theta +d$ (choice depicted in the figure). Therefore the values of the adaptation variable at the time of the spike have the global upper-bound $w_u=\max(v_M,\Phi(w^*)+d)$ or $w_u= \max(v_M,b\,\theta +d)$. This maximal value defines a left boundary of the values of $v$. Indeed, denoting $v_1 =\min\{v, F(v)-w_u+I^*=0\}$ (there are two possible values of $w_1$ because of the convexity of $F$), we have $v\geq v_l:=\min(v_1,v_m)$. This value in turn allows defining the lower $w$ value trajectories can have. As soon as the $w$-nullcline is crossed, $w$ increases. Therefore, $w$ is larger than $w_l:=\min(b\,v_l, w_m)$. It has a typical structure as depicted in figure Fig.~\ref{fig:AlgoPartition}. The set where the algorithm (i) is used is the white space in figure \ref{fig:AlgoPartition}, and is strictly included in $[v_l,\;v_+(I^*-\varepsilon,b)] \times [w_l,w_u]$. 
\end{proof}

Now that we restricted possible initial conditions of the system, we know from lemma \ref{lem:BoundedTrajectories} that the values of $(v,w)$ on the orbits are contained in a bounded set. For the quadratic model, these values depend on the cutoff value $\theta$ since the adaptation variable at the times of the spike diverges, and for models satisfying assumption \ref{Assump:convergence}, it can be defined independently of $\theta$. On this bounded set, the functions $F(v)-w+I$ and $a\,(b\,v-w)$ are both bounded, differentiable with bounded derivatives, which allows definition of time steps $dt$ and phase-space step $dv$ that uniformly ensure a given precision of the numerical algorithm. 

Indeed, in both the time integration stage and the phase space integration stage, the integration method is an Euler scheme, either integrated in time or in the $v$ variable. In these cases, it is well known that the precision of the method is bounded by the second derivative of the integrated variables multiplied by the integration step (the proof of this fact is close from the analysis performed in section \ref{sec:fixedstep}). More specifically, we have:

\begin{enumerate}
	\item The numerical time integration consists of an Euler scheme approximating the solutions of the equations \eqref{eq:GeneralModel} on a compact subset of the phase space where $\vert F(v)-w+I \vert \leq M$. Moreover, the values of the variable $(v,w)$ are contained in a compact set that can be chosen independent of the value of $\theta$ under assumption \ref{Assump:convergence} as already discussed. The second derivative of $v(t)$ and $w(t)$ read:
	\[
	\begin{cases}
		v''(t) = F'(v)(F(v)-w+I) - a(b\,v-w) + I'(t) \\
		w''(t) = a\,(b\,(F(v)-w+I)) -a (b\,v-w)
	\end{cases}
	\]
	where $x'$ denotes the derivative of $x(t)$ with respect to time. We also note that the time-integration is performed on an even more reduced zone of the phase space. Indeed, the fact that $\vert F(v)-w+I\vert <M$ implies that $F(v)\leq M+w_u-I^*$ defining an interval of values of $v$ where the integration is performedm and on this compact set, we clearly have $C_1>0$ and $C_2>0$ such that $\vert a(b\,v-w)\vert <C_1$ and $\vert F'(v) \vert < C_2$. Therefore choosing $dt<\varepsilon / (\max(C_2 M + C_1 + \Vert I' \Vert_{\infty},\vert a b \vert M + C_2))$, we are ensured to have a precision of the order $\varepsilon$ on the evaluated value of $v(t)$ and $w(t)$. Moreover, besides the uniform bound on the error that allows choosing a time-step $dt$ ensuring a precision of order $\varepsilon$ at any point of the phase plane where the time integration is performed, the analysis allows defining an optimal time step $dt$ depending on the point $(v,w)$ and ensuring the desired precision:
	\[\widetilde{dt}(v,w) =\frac{\varepsilon}{\max(\vert v''(t)\vert , \vert w''(t)\vert) }\]
	This varying time step allows achieving the desired precision in a minimal number of operations. 
	
	\item The numerical phase space integration consists of an Euler scheme approximating the solutions of the equations \eqref{eq:WT2} on a compact subset of the phase space. The second derivative of $T(v)$ and $W(v)$ with respect to $v$ (denoted with a double dot) reads:
	\[
	\begin{cases}
		\ddot{W}(v) =  \displaystyle{\frac{a\,b}{F(v)-W+I}-\frac{a\,(b\,v-W)}{(F(v)-W+I)^2} -\frac{a\,F'(v)\,(b\,v-W)}{(F(v)-W+I)^2}} \\
		\qquad\qquad \displaystyle{- \frac{a\,(b\,v-W)\,(a\,(b\,v-W)+I'(T))}{(F(v)-W+I)^3}} \\
		\ddot{T}(v) = \displaystyle{-\frac{F'(v)(F(v)-W+I)-a\,(b\,v-W) + I'(T)}{(F(v)-W+I)^3}}
	\end{cases}
	\]
	and here again, we can find $C_3>0$ and $C_4>0$ such that: $\vert a(b\,v-W)\vert \leq C_3 \vert F(v)-W+I \vert$ and $\vert F'(v) \vert \leq C_4 \vert F(v)-W+I \vert $ and if we defining $dv$ such that:
	\[dv \max \left (C_3\,C_4 + \frac{\vert ab\vert + C_3 + C_3^2}{M} +\frac{\Vert I'\Vert_{\infty} C_3}{M^2} , \; \frac{C_4}{M}+\frac{C_3}{M^2} + \frac{\Vert I'\Vert_{\infty} }{M^3}\right) \leq \varepsilon\]
	we ensure that the error on $W(v)$ and $T(v)$ bounded by $\varepsilon$. Here again, we are in position to choose an optimal integration step depending on the values of the variables $(v,w)$ ensuring the desired precision:
	\[\widetilde{dv}(v,w) = \frac{\varepsilon}{\max(\vert \ddot{W}(v)\vert , \vert \ddot{T}(v)\vert )}.\]
\end{enumerate}

The evaluation of the values of $C_1,\,C_2,\,C_3,\,C_4$ is quite difficult in the most general case. However, as mentioned, we can choose a varying integration steps  $\widetilde{dt}$ and $\widetilde{dv}$ depending on the values of the variables that ensures exactly a precision bounded by $\varepsilon$ at each integration time, as done in the main text.

It is interesting to note that the choice of $M$ impacts the values of $dt$ and $dv$ and therefore the number of operations necessary to integrate the system for a fixed precision $\varepsilon$. The larger $M$ is chosen, the smaller $dt$ is and the larger $dv$ is, meaning that we emphasize the time integration (phase (i) of the algorithm) and give less importance to the phase space simulation (phase (ii) of the algorithm). Conversely, if $M$ is small, the resulting $dv$ will be smaller and $dt$ larger, and the scheme will emphasize the phase-space integration. A balanced choice of $M$ will allow an optimal simulation of the model by taking advantage of both integration methods, in the sense that it will perform the smaller number of operations for a given precision $\varepsilon$. 

\bibliographystyle{apalike}

\end{document}